\documentclass[11pt,reqno]{amsart}

\usepackage{graphicx,booktabs}
\usepackage{amsmath,amssymb,mathrsfs,cite}

\newtheorem{theorem}{Theorem}[section]

\newtheorem{lemma}[theorem]{Lemma}

\theoremstyle{definition}

\theoremstyle{remark}

\numberwithin{equation}{section}

\begin{document}

\title[open cavity scattering problems]{An adaptive finite element
PML method for the open cavity scattering problems}

\author{Yanli Chen}
\address{Department of Mathematics, Northeastern University, Shenyang 110819,
China}
\email{chenyanli@mail.neu.edu.cn}

\author{Peijun Li}
\address{Department of Mathematics, Purdue University, West Lafayette, Indiana
47907, USA}
\email{lipeijun@math.purdue.edu}

\author{Xiaokai Yuan}
\address{School of Mathematical Sciences, Zhejiang University,
Hangzhou 310027, China.}
\email{yuan170@zju.edu.cn}

\thanks{The research of PL is supported in part by the NSF grant
DMS-1912704.}

\subjclass[2010]{}


\begin{abstract}
Consider the electromagnetic scattering of a time-harmonic plane wave by
an open cavity which is embedded in a perfectly electrically conducting infinite
ground plane. This paper is concerned with the numerical solutions of the
transverse electric and magnetic polarizations of the open cavity scattering
problems. In each polarization, the scattering problem is reduced equivalently
into a boundary value problem of the two-dimensional Helmholtz equation in a
bounded domain by using the transparent boundary condition (TBC). An a
posteriori estimate based adaptive finite element method with the perfectly
matched layer (PML) technique is developed to solve the reduced problem. The
estimate takes account both of the finite element approximation error and the
PML truncation error, where the latter is shown to decay exponentially with
respect to the PML medium parameter and the thickness of the PML layer.
Numerical experiments are presented and compared with the adaptive finite
element TBC method for both polarizations to illustrate the competitive behavior
of the proposed method.
\end{abstract}

\maketitle

\section{Introduction}

The phenomena of electromagnetic scattering by open cavities have attracted
much attention due to the significant industrial and military applications in
such areas as antenna synthesis and stealth design. The underlying scattering
problems have been extensively studied by many researchers in the engineering
and applied mathematics communities. We refer to the survey \cite{L18} and the
references cited therein for a comprehensive account on analysis, computation,
and optimal design of the cavity scattering problems.

In applications, one of particular interests is the radar cross section (RCS)
analysis, which aims at how to mitigate or amplify a signal. The RCS is a
quantity which measures the detectability of a target by radar system.
Deliberate control in the form of enhancement or reduction of the RCS of a
target is of high importance in the electromagnetic interference, especially in
the aircraft detection and the stealth design. Since the problems are imposed in
open domains and the solutions may have singularities, it presents challenging
and significant mathematical and computational questions on precise modeling and
accurate computing for the cavity scattering problems in order to successfully
implement any desired control of the RCS. This paper concerns the numerical
solutions of the open cavity scattering problems. We intend to develop an
adaptive finite element method with the perfect matched layer (PML) technique to
overcome the difficulties.

The PML technique was first proposed by B\'{e}renger for solving the
time-dependent Maxwell equations \cite{Berenger94}. Due to its effectiveness,
simplicity and flexibility, the PML technique is widely used in computational
wave propagation \cite{CM98Op, TC97, TY98, CZ17}. It has been
recognized as one of the most important and popular approaches for the domain
truncation. Under
the assumption that the exterior solution is composed of outgoing waves
only, the basic idea of the PML technique is to surround the domain of interest
with a layer of finite thickness of a special medium, which is designed
to either slow down or attenuate all the waves propagating into the PML layer
from inside of the computational domain. As either the PML parameter or the
thickness of the PML layer tends to infinity, the exponential convergence error
estimate was obtained in \cite{HSZ03, LS98} between the solution of the PML
problem and the solution of the Helmholtz-type scattering problem.
The convergence analysis of the PML problems for the three-dimensional
electromagnetic scattering was stuided in \cite{BW05, BP07, CC08, LWZ12}.

In practice, if we use a very thick PML layer and a uniform finite element mesh,
it requires very excessive grids points and hence involves more computational
cost. In contrast, if we choose a thin PML layer, it is inevitable to have a
rapid variation of the PML medium property, which renders a very fine mesh in
order to reach the desired accuracy. On the other hand, the solutions of the
open cavity scattering problem may have singularities due to the existence of
corners of cavities or the discontinuity of the dielectric coefficient for the
filling medium. These singularities slow down the speed of convergence if
uniform mesh refinements are applied. The a posteriori error estimate
based adaptive finite element method is an ideal tool to handle these issues.

A posteriori error estimators are computable quantities in terms of numerical
solutions and data. They measure the error between the numerical solution and
the exact solution without requiring any a priori information of the exact
solution. A reliable a posteriori error estimator plays a crucial role in an
adaptive procedure for mesh modification such as refinement or coarsening. Since
the work of Babu\v{s}ka and Rheinboldt \cite{BR78}, the study of adaptive method
based on a posteriori error estimator has become an active research topic in
scientific computing. Some relevant work can be found in \cite{BW85, AO93, AC91,
CD02, CD01} on the adaptive finite element method. We refer to \cite{M98, MS98,
CL05, CW03} for studies on the scattering problems by using the a posteriori
error estimate based adaptive finite element method.

Motivated by the work of Chen and Liu \cite{CL05}, we develop an adaptive
procedure, which combines the finite element method and the PML technique, to
solve the open cavity scattering problems. Specifically, we consider the
electromagnetic scattering of a time-harmonic plane wave by an open cavity
embedded in an infinite ground plane. The ground plane and the cavity wall are
assumed to be perfect electric conductors. The cavity is assumed to be filled
with some inhomogeneous medium, which may protrude out of the cavity to the
upper half-space in a finite extend. The upper half-space above the ground plane
and the protruding part of the cavity is assumed to be filled with some
homogeneous medium. By assuming invariance of the cavity in the $x_3$ direction,
we consider two fundamental polarizations: transverse magnetic (TM) and
transverse electric (TE) polarizations, where the three-dimensional Maxwell
equations can be reduced to the two-dimensional Helmholtz equation. We restrict
our attention to the numerical solutions of the TM and TE polarizations. In each
polarization, the scattering problem is reduced equivalently into a boundary
value problem of the two-dimensional Helmholtz equation in a bounded domain by
using the transparent boundary condition. Computationally, the PML
technique is utilized to truncate the infinite half-space above the ground plane
and the homogeneous Dirichlet boundary condition is imposed on the outer
boundary of the PML layer. The a posteriori error estimate is deduced between
the solution of the original scattering problem and the finite element solution
of the truncated PML problem. The a posteriori error estimate takes account both
of the finite element discretization error and the truncation error of the PML
method. The PML truncation error has a nice feature of exponential decay in
terms of the PML medium parameter and the thickness of the layer. Based on this
property, the proper PML medium parameter and the thickness of the layer can be
chosen to make the PML error negligible compared with the finite element
discretization error. Once the PML region and the medium property are fixed, the
finite element discretization error is used to design the adaptive strategy.

We point out a closely related work \cite{YBL20}, where an adaptive finite
element method with transparent boundary condition (TBC) was developed for
solving the open cavity scattering problems. Since the nonlocal TBC is
directly used to truncate the open domain, it does not require a layer of
artificially designed absorbing medium to enclose the domain of interest, which
makes the TBC method different from the PML approach. But the TBC is given as
an infinite series and needs to be truncated into a sum of finitely many terms
in computation. Due to the simplicity in the implementation of the PML method,
this work provides a viable alternative to the adaptive finite element TBC
method for solving the open cavity scattering problems. Numerical experiments
are presented and compared with the adaptive finite element TBC method for both
polarizations to illustrate the competitive behavior of the adaptive finite
element PML method.

The outline of this paper is as follows. In Section 2, we introduce the problem
formulation, where the governing equations are given for the TM and TE
polarizations. Sections 3 and 4 are devoted to the analysis of the TM and TE
polarizations, respectively. Topics are organized to address the variational
problem, the PML problem and its convergence, the finite element approximation,
the a posteriori error analysis for the discrete truncated PML problem, and the
adaptive finite element algorithm. In Section 5, some numerical examples are
presented to illustrate the performance of the proposed method. The paper is
concluded with some general remarks in Section 6.

\section{Problem formulation}

Let us first specify the problem geometry which is shown in Figure
\ref{geometry}. Denote by $D\subset \mathbb{R}^2$ the cross section of an
$x_3$-invariant cavity with a Lipschitz continuous boundary $\partial
D=S\cup\Gamma$, where $S$ refers to as the cavity wall and $\Gamma$ is the
opening of the cavity. We assume that the cavity wall $S$ is a perfect electric
conductor and the opening $\Gamma$ is aligned with the perfectly electrically
conducting infinite ground plane $\Gamma_g$. The cavity may be filled with some
inhomogeneous medium, which can be characterized by the dielectric permittivity
$\epsilon$ and the magnetic permeability $\mu$. Moreover, the medium may
protrude from the cavity into the upper half-space. In this case, the cavity is
called an overfilled cavity. Let $B_R^+$ and $B_{\rho}^+$ be the upper
half-discs with radii $R$ and $\rho$, where $\rho>R>0$. Denote by $\Gamma_R^+$
and $\Gamma_{\rho}^+$ the upper semi-circles. The radius $R$ can be chosen
large enough such that the upper half-disc $B_R^+$ can enclose the possibly
protruding inhomogeneous medium from the cavity. The infinite exterior domain
$\mathbb R^2\setminus\overline{B_R^+}$ is assumed to be filled with some
homogeneous medium with a constant dielectric permittivity $\epsilon_0$ and a
constant magnetic permeability $\mu_0$.

\begin{figure}
\centering
\includegraphics[width=0.45\textwidth]{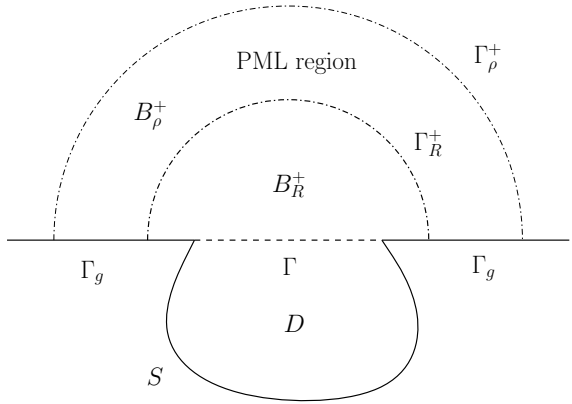}
\caption{Schematic of the open cavity scattering problem.}
\label{geometry}
\end{figure}

Since the structure is invariant the $x_3$-axis, we consider two fundamental
polarizations: transverse magnetic (TM) polarization and transverse electric
(TE) polarization. The three-dimensional Maxwell equations can be reduced to
the two-dimensional Helmholtz equation under these two modes. In the TM
polarization, the magnetic field is transverse to the $x_3$-axis and the
electric field has the form $\boldsymbol{E}(x_1,x_2)=(0,0,u(x_1,x_2))^\top$,
where the scalar function $u$ satisfies
\begin{equation}\label{TM-P1}
\begin{cases}
\Delta u+\kappa^2u=0&\quad\text{in}~\mathbb{R}^2_{+}\cup D,\\
 u=0&\quad\text{on}~\Gamma_g\cup S,
\end{cases}
\end{equation}
where $\kappa=\omega(\varepsilon\mu)^{1/2}$ is the wave number and $\omega>0$ is
the angular frequency. In the TE polarization, the electric field is transverse
to the $x_3$-axis and the magnetic field takes the form
$\boldsymbol{H}(x_1,x_2)=(0,0,u(x_1,x_2))^\top$, where $u$ satisfies
\begin{equation}\label{TE-P1}
\begin{cases}
\nabla\cdot(\kappa^{-2}\nabla u)+u=0&\quad\text{in}~\mathbb{R}^2_{+}\cup D,\\
\partial_{\nu}u=0&\quad\text{on}~\Gamma_g\cup S,
\end{cases}
\end{equation}
where $\nu$ is the unit outward normal vector to $\Gamma_g\cup S$.\par
Consider the incidence of a plane wave
\begin{equation*}
  u^{\rm i}(x_1,x_2)=e^{i(k_1 x_1- k_2 x_2)},
\end{equation*}
which is sent from the above to impinge the cavity. Here $k_1=\kappa_0
\sin\theta, k_2=\kappa_0\cos \theta,\theta\in(-\pi/2,\pi/2)$ is the angle of
the incidence, and $\kappa_0=\omega(\varepsilon_0\mu_0)^{1/2}$ is the wavenumber
in the free space. Due to the perfectly electrically conducting ground plane,
the reflected field in the TM polarization is
\begin{equation*}
  u^{\rm r}(x_1,x_2)=-e^{i(k_1 x_1+ k_2 x_2)},
\end{equation*}
while the reflected field in the TE polarization is
\begin{equation*}
  u^{\rm r}(x_1,x_2)=e^{i(k_1 x_1+ k_2 x_2)}.
\end{equation*}
Let the reference field $u^{\rm ref}$ be the superposition of the incident field
and the reflected field, i.e., $u^{\rm ref}=u^{\rm i}+u^{\rm r}$. The total
field $u$ consists of the reference field $u^{\rm ref}$ and the scattered field
$u^{\rm s}$, i.e.,
\begin{equation*}
  u=u^{\rm ref}+u^{\rm s}.
\end{equation*}
In addition, the scattered field $u^{\rm s}$ is required to satisfy the Sommerfeld radiation condition
\begin{equation}\label{TM-SRC}
  \lim_{r=|x|\rightarrow \infty} r^{1/2}(\partial_r u^{\rm s}-i\kappa_0u^{\rm s})=0.
\end{equation}

\section{TM polarization}

In this section, we consider the TM polarization. First the transparent
boundary condition is introduced to reduce the open cavity problem into a
boundary value problem in a bounded domain. Next the variational problem is
described, and the PML problem and its convergence are discussed. Then the
finite element approximation and the a posteriori error estimate are studied.
Finally the adaptive finite element method with PML is presented for solving
the discrete PML problem.

\subsection{The variational problem}

It can be verified from \eqref{TM-P1} that the scattered field
$u^{\rm s}$ satisfies the Helmholtz equation
\begin{equation}\label{TM-us}
\Delta u^{\rm s}+\kappa_0^2u^{\rm s}=0
\qquad\mathrm{in}~\mathbb{R}^2_{+}\setminus \overline{B_R^+}.
\end{equation}
Based on the radiation condition \eqref{TM-SRC}, we know that the solution of
\eqref{TM-us} has the Fourier series expansion
\begin{equation}\label{TM-usF1}
u^{\rm s}(r,\phi)=\sum_{n=0}^{\infty}\frac{H_n^{(1)}(\kappa_0
r)}{H_n^{(1)}(\kappa_0 R)}(a_n\sin (n\phi)+b_n\cos (n\phi)),\quad r\geq R,
\end{equation}
where $H_n^{(1)}$ is the Hankel function of the first kind with order $n$.
Noting the fact $u=0$ and $u^{\rm ref}=0$ on $\Gamma_g$, we have $u^{\rm
s}(r,0)=u^{\rm s}(r,\pi)=0$, which implies $b_n=0$ and \eqref{TM-usF1} reduces
to
\begin{equation}\label{TM-usF}
u^{\rm s}(r,\phi)=\sum_{n=1}^{\infty}\frac{H_n^{(1)}(\kappa_0 r)}{H_n^{(1)}(\kappa_0 R)}a_n\sin (n\phi),\quad r\geq R.
\end{equation}
Taking the partial derivative of \eqref{TM-usF} with respect to $r$ and evaluating it at $r=R$ yields
\begin{equation}\label{TM-usd}
\partial_r u^{\rm s}(R,\phi)=\kappa_0
\sum_{n=1}^{\infty}\frac{{H_n^{(1)'}}(\kappa_0 R)}{H_n^{(1)}(\kappa_0
R)}a_n\sin (n\phi).
\end{equation}\par
Let $L^2_{\rm TM}(\Gamma_R^+):=\{u\in L^2(\Gamma_R^+):u(R,0)=u(R,\pi)=0\}$. For any $u\in L^2_{\rm TM}(\Gamma_R^+)$, it has the Fourier series expansion
\begin{equation*}
  u(R,\phi)=\sum_{n=1}^{\infty}a_n\sin (n\phi),\quad a_n=\frac{2}{\pi}\int_0^{\pi}u(R,\phi)\sin (n \phi){\rm d}\phi.
\end{equation*}
Define the trace function space $H^s_{\rm TM}(\Gamma_R^+):=\{u\in L^2_{\rm TM}(\Gamma_R^+):\|u\|_{H^s_{\rm TM}(\Gamma_R^+)}\leq \infty\}$, where the $H^s_{\rm TM}(\Gamma_R^+)$ norm is given by
\begin{equation*}
  \|u\|_{H^s_{\rm TM}(\Gamma_R^+)}=\left(\sum_{n=1}^{\infty}(1+n^2)^s|a_n|^2\right)^{1/2}.
\end{equation*}
It is clear that the dual space of $H^s_{\rm TM}(\Gamma_R^+)$ is $H^{-s}_{\rm TM}(\Gamma_R^+)$ with respect to the scalar product in $L^2(\Gamma_R^+)$ given by
\begin{equation*}
  \langle u,v\rangle_{\Gamma_R^+}=\int_{\Gamma_R^+}u\bar{v}{\rm d}s.
\end{equation*}

Introduce the DtN operator
\begin{equation}\label{TM-DtN}
 (\mathscr{B}_{\rm
TM}u)(R,\phi)=\kappa_0\sum_{n=1}^{\infty}\frac{{H_n^{(1)'}}(\kappa_0
R)}{H_n^{(1)}(\kappa_0 R)}a_n\sin (n\phi)\quad\text{on} ~ \Gamma_R^+.
\end{equation}
It is shown in \cite{W06} that the boundary operator $\mathscr{B}_{\rm
TM}:H^{1/2}_{\rm TM}(\Gamma_R^+)\rightarrow H^{-1/2}_{\rm TM}(\Gamma_R^+)$ is
continuous. Using \eqref{TM-usd}--\eqref{TM-DtN}, we obtain the
transparent boundary condition for the scattered field $u^{\rm s}$:
\begin{equation*}
  \partial_r u^{\rm s}=\mathscr{B}_{\rm TM}u^{\rm s}\quad{\rm on}~\Gamma_R^+,
\end{equation*}
which can be equivalently imposed for the total field $u$:
\begin{equation*}
\partial_r u=\mathscr{B}_{\rm TM}u+f\quad {\rm on}~\Gamma_R^+,
\end{equation*}
where $f=\partial_r u^{\rm ref}-\mathscr{B}_{\rm TM}u^{\rm ref}$.

Let $\Omega=B_R^+\cup D$. The open cavity scattering problem can be reduced to
the following boundary value problem:
\begin{equation*}
\begin{cases}
 \Delta u+\kappa^2u=0&\quad\text{in}~\Omega,\\
  u=0&\quad\text{on}~\Gamma_g\cup S,\\
  \partial_r u=\mathscr{B}_{\rm TM}u+f&\quad \text{on}~\Gamma_R^+,
  \end{cases}
\end{equation*}
which has the variational formulation: find $u\in H_S^1(\Omega)=\{u\in
H^1(\Omega):u=0 ~\text{on}~\Gamma_g\cup S\}$ such that
\begin{equation}\label{TM-VF}
a_{\rm TM}(u,v)=\langle f,v\rangle_{\Gamma_R^+}\quad \forall\, v\in
H_S^1(\Omega),
\end{equation}
where the sesquilinear form $a_{\rm TM}(\cdot,\cdot): H^1(\Omega)\times
H^1(\Omega)\to\mathbb C$ is defined by
\begin{equation}\label{TM-SQL}
a_{\rm TM}(u,v)=\int_{\Omega}\left(\nabla u\cdot \nabla \bar{v}-\kappa^2
u\bar{v}\right){\rm d}x-\langle \mathscr{B}_{\rm TM} u,v\rangle_{\Gamma_R^+}.
\end{equation}

The following result states the well-posedness of the variational
problem \eqref{TM-VF}. The proof can be found in \cite{W06}.

\begin{theorem}
The variational problem \eqref{TM-VF} has a unique weak solution in
$H_S^1(\Omega)$, which satisfies the estimate
\begin{equation*}
  \|u\|_{H^1(\Omega)}\leq C\|f\|_{H^{-1/2}_{\rm TM}(\Gamma_R^+)},
\end{equation*}
where $C>0$ is a constant.
\end{theorem}

It follows from the general theory in Babu$\check{\rm s}$ka and Aziz
\cite[Chapter 5]{BA73} that there exists a constant $C>0$ such that the
following inf-sup condition holds:
\begin{equation}\label{TM-is1}
\sup\limits_{0\neq v\in H_S^1(\Omega)}\frac{|a_{\rm
TM}(u, v)|}{\|v\|_{H^1(\Omega)}}\geq
C\|u\|_{H^1(\Omega)}\quad \forall\, u \in H_S^1(\Omega).
\end{equation}

\subsection{The PML problem}

Let $\Omega^{\rm PML}=\{x\in \mathbb{R}^2_+: R<|x|<\rho\}$ be the PML region
which encloses the bounded domain $\Omega$ in the upper half-space. Denote by
$\Omega_{\rho}=B_{\rho}^+\cup D$ the computational domain in which the
truncated PML problem is formulated.

Define the PML parameters by using the complex coordinate stretching
\begin{equation}\label{rtilde-r}
\tilde{r}=\int_0^r\alpha(t){\rm d}t=r\beta(r),
\end{equation}
where $\alpha(r)=1+i\sigma(r)$. In practice, $\sigma$ is usually taken as a
power function
\begin{equation*}
  \sigma(r)=
\begin{cases}
 0, &0\leq r <R,\\
  \sigma_0\big(\frac{r-R}{\rho-R}\big)^m,&r\geq R,
  \end{cases}
\end{equation*}
where $\sigma_0$ is a positive constant and $m\geq 1$ is an integer. It
can be seen from \eqref{rtilde-r} that
\[
\beta(r)=1+i\hat{\sigma}(r), \quad
\hat{\sigma}(r)=\frac{1}{r}\int_R^r\sigma(t){\rm d}t.
\]

In the polar coordinates, the gradient and divergence operators can be
written as
\begin{align}\label{diver}
\nabla u=\partial_r u \boldsymbol{e}_r+\frac{1}{r}\partial_{\phi}u
\boldsymbol{e}_{\phi},\quad
\nabla\cdot
\boldsymbol{u}=\frac{1}{r}\partial_r (ru_r)+\frac{1}{r}\partial_{\phi}u_{\phi},
\end{align}
where $\boldsymbol{u}=u_r\boldsymbol{e}_r+u_{\phi}\boldsymbol{e}_{\phi}$ and
$\boldsymbol{e}_r=(\cos \phi, \sin \phi)^\top,
\boldsymbol{e}_{\phi}=(-\sin \phi,\cos \phi)^\top$. By
the chain rule and \eqref{rtilde-r}, a simple calculation yields
\begin{equation}\label{prt}
  \partial_{\tilde{r}}u=\partial_r u \left(\frac{{\rm d}r}{{\rm d}
\tilde{r}}\right)=\frac{1}{\alpha(r)}\partial_r u.
\end{equation}
Combining \eqref{diver} and \eqref{prt}, we introduce the modified gradient
operator
\begin{equation*}
\tilde{\nabla} u=\frac{1}{\alpha(r)}\partial_ru\boldsymbol{e}_r+\frac{1}{
r\beta(r)}\partial_{\phi}u\boldsymbol{e}_{\phi}.
\end{equation*}
It is easy to verify
\begin{eqnarray*}
 \tilde{\Delta}u=\frac{1}{r
\alpha(r)\beta(r)}\partial_r\left(\frac{r\beta(r)}{\alpha(r)}
\partial_r u\right)+\frac{1}{r\beta(r)}\partial_\phi
\left(\frac{1}{r\beta(r)} \partial_\phi u \right)
   =  \frac{1}{\alpha\beta}\nabla\cdot(A \nabla u),
\end{eqnarray*}
where
\begin{equation*}
A=\begin{bmatrix}
\frac{\beta(r)}{\alpha(r)}\cos^2\phi+\frac{\alpha(r)}
{\beta(r)}\sin^2\phi & \left(\frac{\beta(r)}{\alpha(r)}-\frac{\alpha(r)}{
\beta(r)} \right)\sin\phi\cos\phi\\
\left(\frac{\beta(r)}{\alpha(r)}-\frac{\alpha(r)}{\beta(r)}
\right)\sin\phi\cos\phi
&\frac{\beta(r)}{\alpha(r)}\sin^2\phi+\frac{\alpha(r)}{\beta(r)}\cos^2\phi
 \end{bmatrix}.
\end{equation*}

Hence we obtain the PML equation for the scattered field $u^{\rm
s,PML}$:
\begin{equation*}
\nabla\cdot(A \nabla u^{\rm s,PML})+\kappa_0^2\alpha\beta u^{\rm
s,PML}=0\quad\text{in}~\mathbb{R}^2_{+}\setminus\overline{B_R^+},
\end{equation*}
where $u^{\rm s,PML}$ is required to be uniformly bounded as $r=|x|\to\infty$.
In practice, the open domain $\mathbb{R}^2_{+}\setminus\overline{B_R^+}$ needs
to be truncated into a bounded domain. Replacing $r$ with $\tilde r$ in
\eqref{TM-usF} and noting the exponential decay of the Hankel functions with
a complex argument, we can observe that the scattered field $u^{\rm s,PML}$
decays exponentially in $\mathbb{R}^2_{+}\setminus\overline{B_R^+}$. Hence it
is reasonable to impose the Dirichlet boundary condition
\begin{equation*}
 u^{\rm s,PML}=0\quad {\rm on}~\Gamma_{\rho}^+.
\end{equation*}
We obtain the truncated PML problem
\begin{equation}\label{TM-PML-P1}
\begin{cases}
 \nabla\cdot(A \nabla u^{\rm PML})+\kappa^2\alpha\beta u^{\rm
PML}=F&\quad\text{in}~\Omega_\rho,\\
  u^{\rm PML}=0&\quad\text{on}~\Gamma_g\cup S,\\
   u^{\rm PML}=u^{\rm ref}&\quad \text{on}~\Gamma_\rho^+,
  \end{cases}
\end{equation}
where
\begin{equation*}
  F=
\begin{cases}
\nabla\cdot(A \nabla u^{\rm ref})+\kappa_0^2\alpha\beta u^{\rm
ref}&\mathrm{in}~\Omega^{\rm PML},\\
0& {\rm otherwise}.
 \end{cases}
\end{equation*}

Introduce another DtN operator $\hat{\mathscr{B}}_{\rm TM}: H^{1/2}_{\rm
TM}(\Gamma_R^+)\rightarrow H^{-1/2}_{\rm TM}(\Gamma_R^+)$ which defined as
follows: given $\zeta\in H^{1/2}_{\rm TM}(\Gamma_R^+)$,
\begin{equation*}
  \hat{\mathscr{B}}_{\rm TM} \zeta=\partial_r \xi|_{\Gamma_R^+},
\end{equation*}
where $\xi \in H^1(\Omega^{\rm PML})$ satisfies
\begin{equation*}
\begin{cases}
\nabla\cdot(A \nabla \xi)+\kappa^2\alpha\beta
\xi=0&\quad\text{in}~\Omega^{\rm PML},\\
\xi=\zeta&\quad\text{on}~\Gamma_R^+,\\
\xi =0&\quad \text{on}~\Gamma_g\cup\Gamma_\rho^+.
\end{cases}
\end{equation*}
Using the boundary condition
\begin{equation*}
\partial_r(u^{\rm PML}-u^{\rm ref})|_{\Gamma_R^+}=\hat{\mathscr{B}}_{\rm TM}
(u^{\rm PML}-u^{\rm ref}),
\end{equation*}
and noting $A=I, \alpha=\beta=1$ in $\Omega$, we reformulate
\eqref{TM-PML-P1} equivalently into the following boundary value problem:
\begin{equation}\label{TM-PML-P2}
\begin{cases}
 \Delta u^{\rm PML}+\kappa^2 u^{\rm
PML}=0&\quad\mathrm{in}~\Omega,\\
  u^{\rm PML}=0&\quad\mathrm{on}~\Gamma_g\cup S,\\
   \partial_r u^{\rm PML}=\hat{\mathscr{B}}_{\rm TM}u^{\rm PML}+\hat{f}&\quad
{\rm on}~\Gamma_R^+,
  \end{cases}
\end{equation}
where $\hat{f}=\partial_r u^{\rm ref}-\hat{\mathscr{B}}_{\rm TM}u^{\rm ref}$.
The weak formulation of the problem \eqref{TM-PML-P2} is to find $u^{\rm PML}\in
H_S^1(\Omega)$ such that
\begin{equation}\label{TM-PML-VF}
\hat{a}_{\rm TM}(u^{\rm PML},v)=\langle
\hat{f},v\rangle_{\Gamma_R^+}\quad \forall\, v \in H_S^1(\Omega),
\end{equation}
where the sesquilinear form $\hat{a}_{\rm TM}(\cdot, \cdot): H^1(\Omega)\times
H^1(\Omega)\rightarrow \mathbb{C}$ is defined as
\begin{equation*}
\hat{a}_{\rm TM}(u,v)=\int_{\Omega}\left(\nabla u\cdot \nabla
\bar{v}-\kappa^2 u\bar{v}\right){\rm d}x-\langle
\hat{\mathscr{B}}_{\rm TM} u,v\rangle_{\Gamma_R^+}.
\end{equation*}

\subsection{Convergence of the PML problem}

Consider a boundary value problem of the PML equation in $\Omega^{\rm PML}$:
\begin{equation}\label{TM-PML-P3}
\begin{cases}
 \nabla\cdot(A \nabla w)+\kappa_0^2\alpha\beta w=0&\quad\mathrm{in}~\Omega^{\rm
PML},\\
  w=0&\quad\mathrm{on}~\Gamma_g\cup \Gamma_R^+,\\
   w=q&\quad {\rm on}~\Gamma_{\rho}^+.
  \end{cases}
\end{equation}
Define $H^1_0(\Omega^{\rm PML})=\{u\in H^1(\Omega^{\rm PML}): u=0~\text{on}~
\Gamma_g\cup\Gamma_R^+\cup\Gamma_\rho^+\}$. Given $q\in H^{1/2}_{\rm
TM}(\Gamma_{\rho}^+)$, the weak formulation of \eqref{TM-PML-P3} is to find
$w\in H^1(\Omega^{\rm PML})$ such that $w=0~{\rm on}~ \Gamma_g\cup \Gamma_R^+,
w=q~{\rm on}~ \Gamma_{\rho}^+$ and
\begin{equation}\label{TM-PML-P9}
\hat{b}(w, v)=0 \quad \forall\,v\in H^1_0(\Omega^{\rm PML}),
\end{equation}
where the sesquilinear form $\hat{b}(\cdot, \cdot):H^1(\Omega^{\rm PML})\times
H^1(\Omega^{\rm PML})\to \mathbb{C}$ is
\begin{equation*}
\hat{b}(u, v)=\int_R^{\rho}\int_0^{\pi}\left(\frac{\beta
r}{\alpha}\partial_r u \partial_r \bar{v}+\frac{\alpha}{\beta
r} \partial_\phi u \partial_\phi
\bar{v}-\kappa_0^2\alpha\beta r u \bar{v}\right){\rm
d}r{\rm d}\phi.
\end{equation*}

As is discussed in \cite{CM98}, in general, the uniqueness of \eqref{TM-PML-P9}
can not be guaranteed due to the possible existence of eigenvalues which form a
discrete set. Since our focus is on the convergence analysis, we simply assume
that the PML problem \eqref{TM-PML-P9} has a unique solution in the PML region.

For any $u\in H^1(\Omega^{\rm PML})$, define
\begin{equation*}
\|u\|_{\ast,\Omega^{\rm
PML}}=\left[\int_R^{\rho}\int_0^{\pi}\left(\Big(\frac{
1+\sigma\hat{\sigma}}{ 1+\sigma^2}\Big)r |\partial_r
u|^2+\bigg(\frac{1+\sigma\hat{\sigma}}{1+\hat{\sigma}^2}\bigg)\frac{1}{r}
|\partial_\phi u|^2+
(1+\sigma\hat{\sigma})\kappa_0^2r|u|^2\right){\rm d}r{\rm d}\phi
\right]^{1/2}.
\end{equation*}
It is easy to show that the norm $\|\cdot\|_{\ast,\Omega^{\rm PML}}$ is
equivalent to the usual $H^1(\Omega^{\rm PML})$-norm. An application of the
general theory in \cite[Chapter 5]{BA73} implies that there exists a
positive constant $\hat C$ depending on $\Omega^{\rm PML}$ and $\kappa_0$ such
that
\begin{equation}\label{TM-is2}
\sup_{0\neq v\in H^1_0(\Omega^{\rm
PML})}\frac{|\hat{b}(u, v)|}{\|v\|_{\ast,\Omega^{\rm
PML}}}\geq\hat{C}\|u\|_{\ast,\Omega^{\rm PML}}\quad \forall\,u\in
H^1_0(\Omega^{\rm PML}).
\end{equation}

The following results play an important role in the convergence analysis. The
proof is similar to that of \cite[Theorem 2.4]{CL05} for solving the obstacle
scattering problem and is omitted here for brevity.

\begin{theorem}
There exists a constant $C>0$ independent of $\kappa_0, R, \rho$, and $\sigma_0$
such that the following estimates are satisfied:
\begin{eqnarray}
\label{TM-e1}\||\alpha|^{-1}\nabla w\|_{L^2(\Omega^{\rm PML})}
&\leq &
C\hat{C}^{-1}(1+\kappa_0R)|\alpha_0|\|q\|_{H^{1/2}_{\rm TM}(\Gamma_{\rho}^+)},\\
\label{TM-e2}\left\|\partial_rw\right\|_{H^{-1/2}_{\rm TM}(\Gamma_R^+)}&\leq &
C\hat{C}^{-1}(1+\kappa_0R)^2|\alpha_0|^2\|q\|_{H^{1/2}_{\rm
TM}(\Gamma_{\rho}^+)},
\end{eqnarray}
where $\hat C$ is given in \eqref{TM-is2} and $\alpha_0=1+i\sigma_0$.
\end{theorem}

Following the idea in \cite{LS98}, for any function $f\in H^{1/2}_{\rm
TM}(\Gamma_R^+)$, we introduce the propagation operator $\mathscr
P_{\rm TM}:H^{1/2}_{\rm
TM}(\Gamma_R^+)\rightarrow H^{1/2}_{\rm TM}(\Gamma_{\rho}^+)$ defined by
\begin{equation*}
\mathscr
P_{\rm
TM}(f)=\sum_{n=1}^{\infty}\frac{H_n^{(1)}(\kappa_0\tilde{\rho})}{H_n^{(1)}
(\kappa_0R)}f_n\sin(n\phi),\quad f_n=\frac{2}{\pi}\int_0^{\pi}f(R,\phi)\sin(n
\phi){\rm d}\phi.
\end{equation*}
As shown in \cite{CL05}, the operator $\mathscr P_{\rm TM}:H^{1/2}_{\rm
TM}(\Gamma_R^+)\rightarrow H^{1/2}_{\rm TM}(\Gamma_{\rho}^+)$ is well defined
and satisfies the estimate
\begin{equation}\label{TM-e3}
\|\mathscr P_{\rm TM}(f)\|_{H^{1/2}_{\rm TM}(\Gamma_{\rho}^+)}\leq
e^{-\kappa_0\Im(\tilde{\rho})\left(1-\frac{R^2}{|\tilde{\rho}|^2}\right)^{1/2}}
\|f\|_{H^{1/2}_{\rm TM}(\Gamma_R^+)}\quad \forall\, \rho\geq R.
\end{equation}

\begin{lemma}\label{L:TBCC}
For any $f\in H^{1/2}_{\rm TM}(\Gamma_R^+)$, we have
\begin{equation*}
\|(\mathscr{B}_{\rm TM}-\hat{\mathscr{B}}_{\rm TM})f\|_{H^{-1/2}_{\rm
TM}(\Gamma_R^+)}\leq C\hat{C}^{-1}(1+\kappa_0R)^2|\alpha_0|^2e^{
-\kappa_0\Im(\tilde{\rho} )\left(1-\frac{R^2}{|\tilde{\rho}|^2}\right)^{1/2}}
\|f\|_{H^{1/2}_{\rm TM}(\Gamma_R^+)}.
\end{equation*}
\end{lemma}

\begin{proof}
For any $f\in H^{1/2}_{\rm TM}(\Gamma_R^+)$, it follows the definitions of $\mathscr{B}_{\rm TM}$ and $\hat{\mathscr{B}}_{\rm TM}$ that
\begin{equation*}
(\mathscr{B}_{\rm TM}-\hat{\mathscr{B}}_{\rm TM})f=\partial_rw|_{\Gamma_R^+},
\end{equation*}
where $w\in H^1(\Omega^{\rm PML})$ satisfies
\begin{equation*}
\begin{cases}
 \nabla\cdot(A \nabla w)+\kappa_0^2\alpha\beta w=0&\quad\mathrm{in}~\Omega^{\rm
PML},\\
  w=0&\quad\mathrm{on}~\Gamma_g\cup \Gamma_R^+,\\
   w=\mathscr P_{\rm TM}(f)&\quad {\rm on}~\Gamma_{\rho}^+.
  \end{cases}
\end{equation*}
Using \eqref{TM-e2}--\eqref{TM-e3} yields
\begin{eqnarray*}
\left\|\partial_r w\right\|_{H^{-1/2}_{\rm TM}(\Gamma_R^+)} &\leq&
C\hat{C}^{-1}(1+\kappa_0R)^2|\alpha_0|^2\|\mathscr P_{\rm TM}(f)\|_{H^{1/2}_{\rm
TM}(\Gamma_{\rho}^+)}\\
&\leq& C\hat{C}^{-1}(1+\kappa_0R)^2|\alpha_0|^2e^{-\kappa_0\Im\left(\tilde{\rho}
)(1-\frac{R^2}{|\tilde{\rho}|^2}\right)^{1/2}}\|f\|_{H^{1/2}_{\rm
TM}(\Gamma_R^+)},
\end{eqnarray*}
which completes the proof.
\end{proof}

\begin{theorem}
For sufficiently large $\sigma_0>0$, the PML problem \eqref{TM-PML-VF} has a
unique solution $u^{\rm PML}\in H^1_S(\Omega)$. Moreover, we have the
following estimate:
\begin{equation*}
\|u-u^{\rm PML}\|_{H^1(\Omega)}\leq C\hat{C}^{-1}(1+\kappa_0R)^2|\alpha_0|^2e^{
-\kappa_0\Im(\tilde{\rho})\left(1-\frac{R^2}{|\tilde{\rho}|^2}\right)^{1/2}}\|u^
{\rm PML}-u^{\rm ref}\|_{H^{1/2}_{\rm TM}(\Gamma_R^+)}.
\end{equation*}
\end{theorem}

\begin{proof}
The existence of a unique solution can be shown by following the same arguments
in \cite[Theorem 2.4]{CW03}. Furthermore, by \eqref{TM-VF} and
\eqref{TM-PML-VF}, we have for any $\varphi \in H_S^1(\Omega)$ that
\begin{eqnarray*}
a_{\rm TM}(u-u^{\rm PML}, \varphi) &=&a_{\rm TM}(u, \varphi)-a_{\rm TM}(u^{\rm
PML}, \varphi)\\
&=&\langle f, \varphi\rangle_{\Gamma_R^+}-a_{\rm TM}(u^{\rm PML}, \varphi)  \\
&=& \langle f-\hat{f}, \varphi\rangle_{\Gamma_R^+}+\langle
\hat{f}, \varphi\rangle_{\Gamma_R^+}-a(u^{\rm PML}, \varphi) \\
&=& \langle (\hat{\mathscr{B}}_{\rm TM}-\mathscr{B}_{\rm TM})u^{\rm
ref}, \varphi\rangle_{\Gamma_R^+}+\hat{a}_{\rm TM}(u^{\rm PML}, \varphi)-a_{\rm
TM}(u^{\rm PML}, \varphi) \\
&=&\langle (\mathscr{B}_{\rm TM}-\hat{\mathscr{B}}_{\rm TM})(u^{\rm
PML}-u^{\rm ref}), \varphi\rangle_{\Gamma_R^+},
\end{eqnarray*}
which completes the proof after using Lemma \ref{L:TBCC} and
\eqref{TM-is1}.
\end{proof}

\subsection{Finite element approximation}

Define $ H^1_S(\Omega_{\rho})=\{u\in H^1(\Omega_{\rho}): u=0~{\rm on}
~\Gamma_g\cup S\}$. The weak formulation of \eqref{TM-PML-P1} is to
find $u^{\rm PML}\in H^1_S(\Omega_{\rho})$ and $u^{\rm PML}=u^{\rm
ref}~{\rm on}~\Gamma_{\rho}^+$ such that
\begin{equation}\label{TM-PML-P5}
b(u^{\rm PML}, v)=-\int_{\Omega_{\rho}}F\bar{v}{\rm d}x\quad \forall\,
v\in H^1_S(\Omega_{\rho}),
\end{equation}
where $H^1_0(\Omega_\rho)=\{u\in H^1(\Omega_\rho):
u=0~\text{on}~\Gamma_g\cup S\cup\Gamma_\rho^+\}$ and the
sesquilinear form $b(\cdot, \cdot):H^1(\Omega_{\rho})\times
H^1(\Omega_{\rho})\rightarrow \mathbb{C}$ is given by
 \begin{equation}\label{TM-PML-SQL}
  b(u, v)=\int_{\Omega_{\rho}}(A\nabla u\cdot
\nabla\bar{v}-\kappa^2\alpha\beta u\bar{v}){\rm d}x.
\end{equation}\par
Let $\mathcal{M}_h$ be a regular triangulation of $\Omega_{\rho}$, where $h$ denotes the maximum diameter
of all the elements in $\mathcal{M}_h$. To avoid being distracted from the main focus of
the a posteriori error analysis, we assume for simplicity that $\Gamma_{\rho}^+$
is polygonal to keep from using the isoparametric finite element space and
deriving the approximation error of the boundary $\Gamma_{\rho}^+$.

Let $V_h$ be the a conforming finite element space, i.e.,
\begin{equation*}
  V_h=\{v_h\in C(\bar{\Omega}_{\rho}):v_h|_K\in P_m(K), ~ \forall\, K\in
\mathcal{M}_h\},
\end{equation*}
where $m$ is a positive integer and $P_m(K)$ denotes the set of all polynomials
of degree no more than $m$. The finite element approximation to the variational
problem \eqref{TM-PML-P5} is to find $u_h\in V_h$ with $u_h=u^{\rm ref}~{\rm
on}~\Gamma_{\rho}^+$ such that
\begin{equation}\label{TM-PML-P7}
b(u_h,\psi_h)=-\int_{\Omega_{\rho}}F\bar{\psi}_h{\rm d}x \quad \forall\,
\psi_h\in V_{S, h},
\end{equation}
where $V_{S, h}=\{v_h\in V_h: v_h=0 ~{\rm on}~\Gamma_g\cup S\}$.\par
For sufficiently small $h$, the discrete inf-sup condition of the sesquilinear
form $b$ can be established by an argument of Schatz \cite{S74}. It follows from
the general theory in \cite{BA73} that the truncated variational problem
\eqref{TM-PML-P7} admits a unique solution. Since our focus is the a posteriori
error analysis and the associated adaptive algorithm, we assume that the
discrete problem \eqref{TM-PML-P7} has a unique solution $u_h\in V_h$.

\subsection{A posteriori error analysis}

For any triangular element $K\in \mathcal{M}_h$, denote by $h_K$ its diameter.
Let $\mathcal{B}_h$ denote the set of all the edges that do not lie on $\partial
\Omega_{\rho}$. For any $e\in \mathcal{B}_h$, $h_e$ denotes its length. For any
$K\in \mathcal{M}_h$, we introduce the residual
\begin{equation*}
  R_K(u)=\nabla\cdot(A\nabla u|_K)+\kappa^2\alpha\beta u|_K.
\end{equation*}
For any interior edge $e$, which is the common side of triangular elements
$K_1,K_2\in \mathcal{M}_h$, we define the jump residual across $e$ as
\begin{equation*}
J_e=-(A\nabla u_h|_{K_1}\cdot \nu_1+A\nabla u_h|_{K_2}\cdot \nu_2),
\end{equation*}
where $\nu_j$ is the unit outward normal vector on the boundary of $K_j,j=1,2$. Let
\begin{equation*}
\tilde{R}_K=
\begin{cases}
 R_K(u_h)&\quad{\rm if}~ K\in \mathcal{M}_h\cap \Omega,\\
  R_K(u_h-u^{\rm ref})&\quad {\rm if}~ K\in \mathcal{M}_h\cap \Omega^{\rm PML}.
  \end{cases}
\end{equation*}
For any triangle $K\in \mathcal{M}_h$, denote by $\eta_K$ the local error estimator as follows:
\begin{equation*}
\eta_K=\max_{x\in
K}w(x)\Big(\|h_K\tilde{R}_K\|_{L^2(K)}^2+\frac{1}{2}\sum_{ e\in \partial
K\cap \mathcal{B}_h}\|h_e^{1/2}J_e\|_{L^2(e)}^2\Big)^{1/2},
\end{equation*}
where the rescaling function
\begin{equation*}
w(x)=
\left\{
\begin{array}{ll}
 1&\quad{\rm if}~ x\in \bar{\Omega},\\
|\frac{\alpha}{\alpha_0}|e^{-\kappa\Im{\tilde{r}}\left(1-\frac{r^2}{
|\tilde{r}|^2}\right)^{1/2}}&\quad {\rm if}~ x\in \Omega^{\rm PML}.
  \end{array}\right.
\end{equation*}

For any $\varphi\in H^1(\Omega)$, let $\tilde{\varphi}$ be its extension in
$\Omega^{\rm PML}$ such that
\begin{equation}\label{TM-PML-P8}
\begin{cases}
 \nabla\cdot(\bar {A} \nabla
\tilde{\varphi})+\kappa_0^2\overline{\alpha\beta}\tilde{\varphi}
=0&\quad\mathrm{in}~\Omega^{\rm PML},\\
  \tilde{\varphi}=\varphi&\quad\mathrm{on}~\Gamma_R^+,\\
   \tilde{\varphi}=0&\quad {\rm on}~\Gamma_g\cup \Gamma_{\rho}^+.
  \end{cases}
\end{equation}
Repeating essentially the proofs of those in \cite[Lemmas 4.1 and
4.4]{CL05}, we may obtain the following two results on the extension.

\begin{lemma}\label{TML:L1}
For any $\varphi,\psi\in H^1(\Omega^{\rm PML})$, the following identity holds:
\begin{equation*}
\langle\hat{\mathscr{B}}_{\rm
TM}\varphi,\psi\rangle_{\Gamma_R^+}=\langle\hat{\mathscr{B}}_{\rm
TM}\bar{\psi},\bar{\varphi}\rangle_{\Gamma_R^+}.
\end{equation*}
\end{lemma}

\begin{lemma}\label{TML:L4}
For any $\varphi\in H^1(\Omega)$, let $\tilde{\varphi}$ be its extension in
$H^1(\Omega^{\rm PML})$ according to \eqref{TM-PML-P8}.
Then there exists a constant $C>0$ independent of $\kappa_0, R, \rho$ and
$\sigma_0$ such that
\begin{equation*}
\||\alpha|^{-1}\gamma\nabla\tilde{\varphi}\|_{L^2(\Omega^{\rm PML})}\leq
C\hat{C}^{-1}(1+\kappa_0 R)|\alpha_0|\|\varphi\|_{H^{1/2}(\Gamma_R^+)},
\end{equation*}
where $\gamma(r)=e^{\kappa_0\Im{\tilde{r}}\left(1-\frac{r^2}{|\tilde{r}|^2}\right)^{1/2}}$.
\end{lemma}

The following lemma is needed in order to present the error
representation formula.

\begin{lemma}\label{TML:L2}
For any $\varphi\in H^1(\Omega)$, let $\tilde\varphi$ be its extension in
$H^1(\Omega_{\rho})$ according to \eqref{TM-PML-P8}. Then we have for any
$\xi\in H_{0}^1(\Omega_{\rho})$ that
\begin{equation*}
\int_{\Omega^{\rm PML}}\left(A\nabla\xi\cdot
\nabla\bar{\tilde{\varphi}}-\kappa_0^2\alpha\beta\xi\bar{\tilde{\varphi}}
\right){\rm d}x=-\langle\hat{\mathscr{B}}_{\rm
TM}\xi,\varphi\rangle_{\Gamma_R^+}.
\end{equation*}
\end{lemma}

\begin{proof}
Multiplying the first equation of \eqref{TM-PML-P8} by $\xi\in
H_{0}^1(\Omega_{\rho})$, using the integration by parts, and noting $A=I$ on
$\Gamma_R^+$, we deduce
\begin{eqnarray*}
\int_{\Omega^{\rm PML}}(\bar{A}\nabla \tilde{\varphi}\cdot\nabla
\bar{\xi}-\kappa_0^2\overline{\alpha\beta}\tilde{\varphi}\bar{\xi}){\rm
d}x=\int_{\partial \Omega^{\rm PML}} (\bar{A}\nabla \tilde{\varphi})\cdot
\nu\bar{\xi}{\rm d}s
= -\int_{\Gamma_R^+}\partial_{\nu}\tilde{\varphi} \bar{\xi}{\rm d}s,
\end{eqnarray*}
where $\nu$ is the outward normal vector to $\Gamma_R^+$ pointing to the
outside of $\Omega$. Taking the complex conjugate on both sides of the above
equation yields
\begin{equation*}
\int_{\Omega^{\rm PML}}(A \nabla \xi  \cdot\nabla
\bar{\tilde{\varphi}}-\kappa_0^2\alpha\beta\xi\bar{\tilde{\varphi}}){\rm
d}x=-\int_{\Gamma_R^+}\partial_{\nu} \bar{\tilde{\varphi}} \xi{\rm d}s.
\end{equation*}
It follows from the definition of $\hat{\mathscr{B}}_{\rm TM}:H^{1/2}(\Gamma_R^+)\rightarrow H^{-1/2}(\Gamma_R^+)$ that
\begin{equation*}
  \partial_{\nu}\bar{\tilde{\varphi}} |_{\Gamma_R^+}=\hat{\mathscr{B}}_{\rm TM}\bar{\varphi}.
\end{equation*}
Combining the above two equations leads to
\begin{equation*}
\int_{\Omega^{\rm PML}}(A \nabla \xi  \cdot\nabla
\bar{\tilde{\varphi}}-\kappa_0^2\alpha\beta\xi\bar{\tilde{\varphi}}){\rm
d}x=-\langle\hat{\mathscr{B}}_{\rm TM}\bar{\varphi},\bar{\xi}\rangle
_{\Gamma_R^+}.
\end{equation*}
By Lemma \ref{TML:L1}, we have
\begin{equation*}
\int_{\Omega^{\rm PML}}(A \nabla \xi  \cdot\nabla
\bar{\tilde{\varphi}}-\alpha\beta\kappa^2\xi\bar{\tilde{\varphi}}){\rm
d}x=-\langle\hat{\mathscr{B}}_{\rm TM}\xi,\varphi\rangle _{\Gamma_R^+},
\end{equation*}
which completes the proof.
\end{proof}

The following lemma gives the error representation formula.

\begin{lemma}[error representation formula]\label{TML:L3}
For any $\varphi\in H^1(\Omega)$, let $\tilde\varphi$ be its extension in
$H^1(\Omega_{\rho})$ according to \eqref{TM-PML-P8}. For any
$\varphi_h\in V_{S, h}$, the following identity holds:
\begin{eqnarray*}
a_{\rm TM}(u-u_h,\varphi)&=&\langle \mathscr{B}_{\rm TM}(u_h-u^{\rm
ref}) -\hat{\mathscr{B}}_{\rm TM}(u_h-u^{\rm
ref}),\varphi\rangle_{\Gamma_R^+}-b(u_h,\varphi-\varphi_h)\\\label{TM-eq2}
 &&-\int_{\Omega^{\rm PML}}\big(\nabla\cdot(A \nabla u^{\rm
ref})+\kappa_0^2\alpha\beta
u^{\rm ref}\big)(\bar{\tilde{\varphi}}-\bar{\varphi}_h){\rm d}x.
\end{eqnarray*}
\end{lemma}

\begin{proof}
It follows from \eqref{TM-VF} that
\begin{eqnarray}
\nonumber a_{\rm TM}(u-u_h,\varphi)&=&a_{\rm TM}(u,\varphi)-a_{\rm
TM}(u_h,\varphi)  \\\label{TM-eq9}
  &=&\langle f, \varphi\rangle_{\Gamma_R^+}-b(u_h,\varphi-\varphi_h)+b(u_h,\varphi)-b(u_h,\varphi_h)-a_{\rm TM}(u_h,\varphi).
\end{eqnarray}
Using \eqref{TE-PML-P7} and the integration by parts, we obtain
\begin{eqnarray}
\nonumber b(u_h,\varphi_h) &=&- \int_{\Omega^{\rm PML}}F\bar{\varphi}_h{\rm
d}x\\
 \nonumber  &=&-\int_{\Omega^{\rm PML}}\big(  \nabla\cdot(A \nabla u^{\rm
ref})+\kappa_0^2\alpha\beta u^{\rm ref}\big)\bar{\varphi}_h{\rm d}x\\
\nonumber  &=&\int_{\Omega^{\rm PML}}\big(\nabla\cdot(A \nabla u^{\rm
ref})+\kappa_0^2\alpha\beta
u^{\rm ref}\big)(\bar{\tilde{\varphi}}-\bar{\varphi}_h){\rm
d}x+\int_{\Omega^{\rm PML}}\big( A \nabla u^{\rm ref}\cdot\nabla
\bar{\tilde{\varphi}} +\kappa_0^2\alpha\beta u^{\rm
ref}\bar{\tilde{\varphi}}\big){\rm d}x\\\label{TM-eq5}
&&+\int_{\Gamma_R^+}\partial_{\nu}u^{\rm ref} \bar{\varphi}{\rm d}s.
\end{eqnarray}
By the definition of the sesquilinear form \eqref{TM-PML-SQL}, we have
\begin{eqnarray}
 b(u_h,\varphi) 
   &=& \int_{\Omega}\left(A \nabla u_h  \cdot\nabla \bar{\varphi}-\kappa^2\alpha\beta u_h\bar{\varphi}\right){\rm d}x+\int_{\Omega^{\rm PML}}\left(A \nabla u_h  \cdot\nabla \bar{\tilde{\varphi}}-\kappa_0^2\alpha\beta u_h\bar{\tilde{\varphi}}\right){\rm d}x.\label{TM-eq6}
\end{eqnarray}
It is easy to get from \eqref{TM-SQL} that
\begin{equation}\label{TM-eq7}
a_{\rm TM}(u_h,\varphi)= \int_{\Omega}\left(A \nabla u_h  \cdot\nabla \bar{\varphi}-\kappa^2\alpha\beta u_h\bar{\varphi}\right){\rm d}x- \langle \mathscr{B}_{\rm TM}u_h, \varphi\rangle_{\Gamma_R^+}.
\end{equation}
Using \eqref{TM-eq5}--\eqref{TM-eq7} yields
\begin{eqnarray*}
   &&b(u_h,\varphi)-b(u_h,\varphi_h)-a_{\rm TM}(u_h,\varphi)
   =-\int_{\Omega^{\rm PML}}\big(  \nabla\cdot(A \nabla u^{\rm
ref})+\kappa_0^2\alpha\beta u^{\rm
ref}\big)(\bar{\tilde{\varphi}}-\bar{\varphi}_h){\rm d}x\\
   &&+\int_{\Omega^{\rm PML}}\big(A \nabla (u_h-u^{\rm ref})\cdot\nabla
\bar{\tilde{\varphi}}-\alpha\beta\kappa^2(u_h-u^{\rm
ref})\bar{\tilde{\varphi}}\big){\rm d}x
 -\int_{\Gamma_R^+}\partial_{\nu}u^{\rm ref} \bar{\varphi}{\rm d}s+ \langle
\mathscr{B}_{\rm TM}u_h, \varphi\rangle_{\Gamma_R^+},
\end{eqnarray*}
which together with Lemma \ref{TML:L2} implies
\begin{eqnarray}
 \nonumber  &&b(u_h,\varphi)-b(u_h,\varphi_h)-a_{\rm TM}(u_h,\varphi)  \\
\nonumber   &=&-\int_{\Omega^{\rm PML}}\big(  \nabla\cdot(A \nabla u^{\rm
ref})+\kappa_0^2\alpha\beta
u^{\rm ref}\big)(\bar{\tilde{\varphi}}-\bar{\varphi}_h){\rm d}x\\
 &&+\langle \mathscr{B}_{\rm TM}u_h-\hat{\mathscr{B}}_{\rm TM}u_h,
\varphi\rangle_{\Gamma_R^+} +\langle-\partial_{\nu}u^{\rm
ref}+\hat{\mathscr{B}}_{\rm TM}u^{\rm
ref},\varphi\rangle_{\Gamma_R^+}.\label{TM-eq8}
\end{eqnarray}
Substituting \eqref{TM-eq8} into \eqref{TM-eq9}, we have
\begin{eqnarray*}
a_{\rm TM}(u-u_h,\varphi) &=&\langle \partial_{\nu}u^{\rm ref}-\mathscr{B}_{\rm
TM}u^{\rm ref},\varphi\rangle_{\Gamma_R^+}-b(u_h,\varphi-\varphi_h)\\
&&-\int_{\Omega^{\rm PML}}\big(  \nabla\cdot(A \nabla u^{\rm
ref})+\kappa_0^2\alpha\beta u^{\rm
ref}\big)(\bar{\tilde{\varphi}}-\bar{\varphi}_h){\rm d}x\\
&&+\langle \mathscr{B}_{\rm TM}u_h-\hat{\mathscr{B}}_{\rm TM}u_h,
\varphi\rangle_{\Gamma_R^+}+\langle-\partial_{\nu} u^{\rm
ref}+\hat{\mathscr{B}}_{\rm TM}u^{\rm ref},\varphi\rangle_{\Gamma_R^+}\\
&=&\langle \mathscr{B}_{\rm TM}(u_h-u^{\rm ref}) -\hat{\mathscr{B}}_{\rm
TM}(u_h-u^{\rm ref}),\varphi\rangle_{\Gamma_R^+}-b(u_h,\varphi-\varphi_h)\\
&&-\int_{\Omega^{\rm PML}}\big(\nabla\cdot(A \nabla
u^{\rm ref})+\kappa_0^2\alpha\beta u^{\rm
ref}\big)(\bar{\tilde{\varphi}}-\bar{\varphi}_h){\rm d}x,
\end{eqnarray*}
which completes the proof.
\end{proof}

Let $\Pi_h: H^1_S(\Omega_{\rho})\to V_{S, h}$ be the Clement-type
interpolation operator. It can be verified that the operator enjoys the
following estimates: for any $v \in H^1_S(\Omega_{\rho})$,
\begin{equation*}
\|v-\Pi_h v\|_{L^2(K)}\leq C h_K\|\nabla v\|_{L^2(\tilde{K})},\quad \|v-\Pi_h
v\|_{L^2(e)}\leq C h_e^{1/2}\|\nabla v\|_{L^2(\tilde{e})},
\end{equation*}
where $\tilde{K}$ and $\tilde{e}$ are the union of all elements in
$\mathcal{M}_h$ having nonempty intersection with $K\in \mathcal{M}_h$ and the
side $e$, respectively.

The following theorem presents the a posteriori error estimate and is the main result for the TM
polarization.

\begin{theorem}\label{T:T1}
Let $u$ and $u_h$ be the solutions of \eqref{TM-VF} and \eqref{TM-PML-P7},
respectively. There exists a constant $C$ depending only on the minimum angle of
the mesh $\mathcal{M}_h$ such that the following a posterior error estimate
holds:
\begin{eqnarray*}
  \|u-u_h\|_{H^1(\Omega)}&\leq & C\hat{C}^{-1}(1+\kappa
R)\bigg(\sum_{K\in\mathcal{M}_h}\eta_K^2\bigg)^{1/2}\\
& & + C\hat{C}^{-1}(1+\kappa_0R)^2|\alpha_0|^2e^{-\kappa_0\Im{(\tilde{\rho})}
(1-\frac{R^2}{|\tilde{\rho}|^2})^{1/2}}\|u_h-u^{\rm ref}\|_{H^{1/2}_{\rm
TM}(\Gamma_R^+)} .
\end{eqnarray*}
\end{theorem}

\begin{proof}
Taking $\varphi_h=\Pi_h \varphi$ and using Lemma \ref{TML:L3}, we have
\begin{eqnarray*}
  a_{\rm TM}(u-u_h,\varphi)&=&\langle \mathscr{B}_{\rm TM}(u_h-u^{\rm
ref}) -\hat{\mathscr{B}}_{\rm TM}(u_h-u^{\rm
ref}),\varphi\rangle_{\Gamma_R^+}-b(u_h,\varphi-\Pi_h \varphi)\\\label{TM-eq3}
 & &-\int_{\Omega^{\rm PML}}\big(  \nabla\cdot(A \nabla u^{\rm
ref})+\kappa_0^2\alpha\beta u^{\rm ref}\big)(\bar{\tilde{\varphi}}-\Pi_h
\bar{\varphi}){\rm d}x\\
   &:= &{\rm I}_1+{\rm I}_2+{\rm I}_3.
\end{eqnarray*}
It follows from Lemma \ref{L:TBCC} that
\begin{eqnarray*}
   {\rm I}_1 &=& \langle \mathscr{B}_{\rm TM}(u_h-u^{\rm ref})
-\hat{\mathscr{B}}_{\rm
TM}(u_h-u^{\rm ref}),\varphi\rangle_{\Gamma_R^+} \\
&\leq  & C\hat{C}^{-1}(1+\kappa_0R)^2|\alpha_0|^2e^{-\kappa_0\Im{(\tilde{\rho})}
\left(1-\frac{R^2}{|\tilde{\rho}|^2}\right)^{1/2}}\|u_h-u^{\rm
ref}\|_{H^{1/2}_{\rm TM}(\Gamma_R^+)}\|\varphi\|_{H^{1/2}_{\rm TM}(\Gamma_R^+)}.
\end{eqnarray*}
Using the integration by parts yields
\begin{eqnarray*}
  {\rm I}_2+{\rm I}_3 &=&\sum_{K\in\mathcal{M}_h\cap \Omega}\bigg(\int_K
R_K(u_h-u^{\rm
ref})(\bar{\varphi}-\Pi_h \bar{\varphi}){\rm d}x+\sum_{e\in\partial K\cap
\mathcal{B}_h}\frac{1}{2}\int_e J_e (\bar{\varphi}-\Pi_h \bar{\varphi}){\rm
d}s\bigg)\\
  & & +  \sum_{K\in\mathcal{M}_h\cap \Omega^{\rm PML}}\bigg(\int_K
R_K(u_h)(\bar{\varphi}-\Pi_h \bar{\varphi}){\rm d}x+\sum_{e\in\partial K\cap
\mathcal{B}_h}\frac{1}{2}\int_e J_e (\bar{\varphi}-\Pi_h \bar{\varphi}){\rm
d}s\bigg).
\end{eqnarray*}
It follows from the Cauchy--Schwarz inequality, the interpolation estimates and
lemma \ref{TML:L4} that
\begin{eqnarray*}
  |{\rm I}_2+{\rm I}_3|&\leq & C\sum_{K\in\mathcal{M}_h}\bigg(\|h_K
\tilde{R}_K\|^2_{L^2(K)}+\frac{1}{2}\sum_{e\in\partial K\cap
\mathcal{B}_h}\|h_e^{1/2}J_e\|^2_{L^2(e)} \bigg)^{1/2}\|\nabla
\varphi\|_{L^2(\tilde{K})}\\
&\leq & C \sum_{K\in\mathcal{M}_h} \eta_K\|w^{-1}\nabla
\varphi\|_{L^2(\tilde{K})}\\
&\leq & C\hat{C}^{-1}(1+\kappa
R)\bigg(\sum_{K\in\mathcal{M}_h}\eta_K^2\bigg)^{1/2} \|\varphi\|_{H^{1/2}
(\Gamma_R^+)}.
\end{eqnarray*}
Using the inf-sup condition \eqref{TM-is1} and combining the above estimates, we get
\begin{eqnarray*}
  \|u-u_h\|_{H^1(\Omega)} &\leq &  C\sup_{0\neq\varphi\in
H^1_S(\Omega)}\frac{|a_{\rm TM}(u-u_h,\varphi)}{\|\varphi\|_{H^1(\Omega)}} \\
   &\leq & C\hat{C}^{-1}(1+\kappa
R)\bigg(\sum_{K\in\mathcal{M}_h}\eta_K^2\bigg)^{1/2}\\
& & +
C\hat{C}^{-1}(1+\kappa_0R)^2|\alpha_0|^2e^{-\kappa_0\Im{(\tilde{\rho})}
\left(1-\frac{R^2}{|\tilde{\rho}|^2}\right)^{1/2}}\|u_h-u^{\rm
ref}\|_{H^{1/2}_{\rm TM}(\Gamma_R^+)},
\end{eqnarray*}
which completes the proof.
\end{proof}

\subsection{Adaptive FEM algorithm}

It can be seen from the Theorem \ref{T:T1} that the a posteriori error estimate
consists of two parts: the finite element approximation error $\varepsilon_h$
and the truncation error of the PML method $\varepsilon_{\rm PML}$, where
\begin{equation*}
 \varepsilon_h= \bigg(\sum_{K\in\mathcal{M}_h}\eta_K^2\bigg)^{1/2},\quad
\varepsilon_{\rm
PML}=e^{-\kappa_0\Im{(\tilde{\rho})}(1-\frac{R^2}{|\tilde{\rho}|^2} )^ {
1/2}}\|u_h-u^{\rm ref}\|_{H^{1/2}_{\rm TM}(\Gamma_R^+)} .
\end{equation*}
In the implementation, we may first choose $\sigma_0$ and $\rho$ to make sure
that the PML error $\varepsilon_{\rm PML}$ is small enough, for
instance $\varepsilon_{\rm PML}\leq 10^{-8}$, such that the PML error is
negligible compared with the finite element approximation error. Next we design
the adaptive strategy to modify the mesh according to the estimate
$\varepsilon_h$. Table \ref{Tab1} shows the algorithm of the adaptive finite
element PML method for solving the open cavity scattering problem in the TM
polarization.

\begin{table}[ht]
\caption{The adaptive finite element PML method for TM polarization.}\label{Tab1}
\begin{tabular}{cp{.8\textwidth}}
\toprule
(1) & Given the tolerance $\varepsilon>0$ and the parameter $\tau\in(0, 1)$.\\
(2) & Choose $\sigma_0$ and $\rho$ such that $\varepsilon_{\rm PML}\leq
10^{-8}$.\\
(3) & Construct an initial triangulation $\mathcal{M}_h$ over $\Omega_{\rho}$
 and compute error estimators.\\
(4) & While $\varepsilon_h>\varepsilon$ do\\
(5) &\quad\quad\quad refine $\mathcal{M}_h$ according to the strategy\\
(6) &\quad\quad\quad \quad\quad if $\eta_{\hat{K}}>\tau \max\limits_{K\in
\mathcal{M}_h}\eta_K$, refine the element $\hat{K}\in \mathcal{M}_h$;\\
(7) &\quad\quad\quad obtain a new mesh denoted still by $\mathcal{M}_h$;\\
(8) &\quad\quad\quad solve \eqref{TM-PML-P7} on the new mesh $\mathcal{M}_h$ and
compute the error estimators.\\
(9) & End while.\\
\bottomrule
\end{tabular}
\end{table}

\section{TE polarization}

In this section, we consider the TE polarization. Since the discussions are
similar to the TM polarization, we briefly present the parallel results without
providing the details.

\subsection{Variational problem}

It can be verified from \eqref{TE-P1} that the scattered field $u^{\rm s}$
satisfies the Helmholtz equation
\begin{equation}\label{TE-us}
\Delta u^{\rm s}+\kappa_0^2u^{\rm s}=0\quad \mathrm{in}~\mathbb{R}_+^2\setminus
\overline{B_R^+}.
\end{equation}
By the radiation condition \eqref{TM-SRC}, the solution of \eqref{TE-us} has
the Fourier series expansion
\begin{equation}\label{TE-usF1}
u^{\rm s}(r,\phi)=\sum_{n=0}^{\infty}\frac{H_n^{(1)}(\kappa_0 r)}{H_n^{(1)}(\kappa_0 R)}(a_n\sin (n\phi)+b_n\cos (n\phi)),\quad r\geq R.
\end{equation}
Using the fact $\partial_{\nu}u=0$ and $\partial_{\nu}u^{\rm ref}=0$ on
$\Gamma_g$, we have $\partial_{\phi}u^{\rm s}(r,0)=\partial_{\phi}u^{\rm
s}(r,\pi)=0$. Hence $a_n=0$ and \eqref{TE-usF1} reduces to
\begin{equation}\label{TE-usF}
u^{\rm s}(r,\phi)=\sum_{n=0}^{\infty}\frac{H_n^{(1)}(\kappa_0
r)}{H_n^{(1)}(\kappa_0 R)}b_n\cos (n\phi),\quad r\geq R,
\end{equation}
which gives
\begin{equation*}
\partial_r u^{\rm s}(R,\phi)=\kappa_0
\sum_{n=0}^{\infty}\frac{{H_n^{(1)'}}(\kappa_0 R)}{H_n^{(1)}(\kappa_0
R)}b_n\cos(n\phi).
\end{equation*}

Let $L^2_{\rm TE}(\Gamma_R^+):=\{u\in L^2(\Gamma_R^+):\partial_{\phi}u(R,0)=\partial_{\phi}u(R,\pi)=0\}$. For any $u\in L^2_{\rm TE}(\Gamma_R^+)$, it has the Fourier series expansion
\begin{equation*}
  u(R,\phi)=\sum_{n=0}^{\infty}b_n\cos (n\phi),
\end{equation*}
where
\begin{equation*}
   b_0=\frac{1}{\pi}\int_0^{\pi}u(R,\phi){\rm d}\phi,\quad b_n=\frac{2}{\pi}\int_0^{\pi}u(R,\phi)\cos (n \phi){\rm d}\phi.
\end{equation*}
Define the trace function space $H^s_{\rm TE}(\Gamma_R^+):=\{u\in L^2_{\rm TE}(\Gamma_R^+):\|u\|_{H^s_{\rm TE}(\Gamma_R^+)}\leq \infty\}$, where the $H^s_{\rm TE}(\Gamma_R^+)$ norm is given by
\begin{equation*}
  \|u\|_{H^s_{\rm TE}(\Gamma_R^+)}=\left(\sum_{n=0}^{\infty}(1+n^2)^s|b_n|^2\right)^{1/2}.
\end{equation*}
It is clear that the dual space of $H^s_{\rm TE}(\Gamma_R^+)$ is $H^{-s}_{\rm TE}(\Gamma_R^+)$ with respect to the scalar product in $L^2(\Gamma_R^+)$ given by
\begin{equation*}
  \langle u,v\rangle_{\Gamma_R^+}=\int_{\Gamma_R^+}u\bar{v}{\rm d}s.
\end{equation*}

We introduce a DtN operator on $\Gamma_R^+$:
\begin{equation}\label{TE-DtN}
  (\mathscr{B}_{\rm
TE}u)(R,\phi)=\kappa_0\sum_{n=0}^{\infty}\frac{{H_n^{(1)'}}(\kappa_0
R)}{H_n^{(1)}(\kappa_0 R)}b_n\cos (n\phi).
\end{equation}
It is shown \cite[Lemma 3.1]{W06} that the DtN operator $\mathscr{B}_{\rm
TE}:H^{1/2}_{\rm TE}(\Gamma_R^+)\rightarrow H^{-1/2}_{\rm TE}(\Gamma_R^+)$ is
continuous.
Using the boundary operator \eqref{TE-DtN}, we obtain the transparent boundary
condition for the TE polarization:
\begin{equation*}
  \partial_r u^{\rm s}=\mathscr{B}_{\rm TE}u^{\rm s}\quad\text{on}~ \Gamma_R^+,
\end{equation*}
which can be equivalently written for the total field $u$:
\begin{equation*}
\partial_r u=\mathscr{B}_{\rm TE}u+g\quad {\rm on}~\Gamma_R^+,
\end{equation*}
where $g=\partial_r u^{\rm ref}-\mathscr{B}_{\rm TE}u^{\rm ref}$.

In the TE polarization, the open cavity scattering problem can be reduced to the following boundary value problem:
\begin{equation*}
\left\{
\begin{array}{ll}
 \nabla\cdot(\kappa^{-2}\nabla u)+u=0&\qquad\mathrm{in}~\Omega,\\
  \partial_{\nu}u=0&\qquad\mathrm{on}~\Gamma_g\cup S,\\
  \partial_r u=\mathscr{B}_{\rm TE}u+g&\qquad {\rm on}~\Gamma_R^+,
  \end{array}\right.
\end{equation*}
which has the variational formulation: find $u\in H^1(\Omega)$ such that
\begin{equation}\label{TE-VF}
a_{\rm TE}(u,v)=\langle \kappa_0^{-2}g,v\rangle_{\Gamma_R^+}\quad \forall\, v\in
H^1(\Omega).
\end{equation}
Here the sesquilinear form $ a_{\rm TE}(\cdot,\cdot): H^1(\Omega)\times
H^1(\Omega)\to\mathbb C$ is given by
\begin{equation*}
a_{\rm TE}(u,v)=\int_{\Omega}\left(\kappa^{-2}\nabla u\cdot \nabla
\bar{v}-u\bar{v}\right){\rm d}x-\langle \kappa_0^{-2}\mathscr{B}_{\rm TE}
u,v\rangle_{\Gamma_R^+}.
\end{equation*}

The following theorem concerns the well-posedness for the variational problem
\eqref{TE-VF} and the proof can be found in \cite{L18}.

\begin{theorem}
The variational problem \eqref{TE-VF} has a unique weak solution in $H^1(\Omega)$, which satisfies the estimate
\begin{equation*}
  \|u\|_{H^1(\Omega)}\lesssim \|g\|_{H^{-1/2}_{\rm TE}(\Gamma_R^+)}.
\end{equation*}
\end{theorem}

The general theory in Babu\v{s}ka and Aziz \cite[Chapter 5]{BA73}
implies that there exists a constant $C>0$ such that the following inf-sup
condition holds:
\begin{equation}\label{TE-is1}
 \sup\limits_{0\neq v\in H^1(\Omega)}\frac{|a_{\rm
TE}(u, v)|}{\|v\|_{H^1(\Omega)}}\geq
C\|u\|_{H^1(\Omega)}\quad \forall\, u \in H^1(\Omega).
\end{equation}

\subsection{The PML problem}

Using the complex coordinate stretching \eqref{rtilde-r}, we may similarly
obtain the truncated PML problem in the TE polarization:
\begin{equation}\label{TE-PML-P1}
\left\{
\begin{array}{ll}
 \nabla\cdot(\kappa^{-2}A \nabla u^{\rm PML})+\alpha\beta u^{\rm
PML}=G&\quad\mathrm{in}~\Omega_\rho,\\
  (A \nabla u^{\rm PML})\cdot \nu=0&\quad\mathrm{on}~\Gamma_g\cup S,\\
   u^{\rm PML}=u^{\rm ref}&\quad {\rm on}~\Gamma_\rho^+,
  \end{array}\right.
\end{equation}
where
\begin{equation*}
  G=\left\{
\begin{array}{ll}
\nabla\cdot(\kappa_0^{-2}A \nabla u^{\rm ref})+\alpha\beta u^{\rm ref}&\mathrm{in}~\Omega^{\rm PML},\\
0& {\rm otherwise}.
 \end{array}\right.
\end{equation*}

A DtN operator $\hat{\mathscr{B}}_{\rm TE}: H^{1/2}_{\rm
TE}(\Gamma_R^+)\rightarrow H^{-1/2}_{\rm TE}(\Gamma_R^+)$ is defined as follows:
given $f\in H^{1/2}_{\rm TE}(\Gamma_R^+)$,
\begin{equation*}
  \hat{\mathscr{B}}_{\rm TE} f=\partial_r  \xi|_{\Gamma_R^+},
\end{equation*}
where $\xi \in H^1(\Omega^{\rm PML})$ satisfies
\begin{equation*}
\left\{
\begin{array}{ll}
 \nabla\cdot(\kappa^{-2}A \nabla \xi)+\alpha\beta
\xi=0&\quad\mathrm{in}~\Omega^{\rm PML},\\
  \xi=f&\quad\mathrm{on}~\Gamma_R^+,\\
   \xi=0&\quad {\rm on}~\Gamma_\rho^+,\\
   (A \nabla \xi)\cdot \nu=0&\quad {\rm on}~\Gamma_g.
  \end{array}\right.
\end{equation*}
By imposing the boundary condition
\begin{equation*}
\partial_r(u^{\rm PML}-u^{\rm ref})=\hat{\mathscr{B}}_{\rm TE}
(u^{\rm PML}-u^{\rm ref})\quad\text{on}~ \Gamma_R^+,
\end{equation*}
the problem \eqref{TE-PML-P1} can be reformulated as
\begin{equation}\label{TE-PML-P2}
\left\{
\begin{array}{ll}
 \nabla\cdot(\kappa^{-2}A \nabla u^{\rm PML})+\alpha\beta u^{\rm
PML}=0&\quad\mathrm{in}~\Omega,\\
 (A \nabla u^{\rm PML})\cdot \nu=0&\quad\mathrm{on}~\Gamma_g\cup S,\\
   \partial_r u^{\rm PML}=\hat{\mathscr{B}}_{\rm TE}u^{\rm PML}+\hat{g}&\quad
{\rm on}~\Gamma_R^+,
  \end{array}\right.
\end{equation}
where $\hat{g}=\partial_r u^{\rm ref}-\hat{\mathscr{B}}_{\rm TE}u^{\rm ref}$.
The weak formulation of the problem \eqref{TE-PML-P2} is to find $u^{\rm PML}\in
H^1(\Omega)$ such that
\begin{equation}\label{TE-PML-VF}
\hat{a}_{\rm TE}(u^{\rm PML},v)=\langle
\kappa_0^{-2}\hat{g},v\rangle_{\Gamma_R^+}\quad \forall\, v\in H^1(\Omega),
\end{equation}
were the sesquilinear form $\hat{a}_{\rm TE}(\cdot,\cdot): H^1(\Omega)\times
H^1(\Omega)\rightarrow \mathbb{C}$ is defined as
\begin{equation*}
  \hat{a}_{\rm TE}(u,v)=\int_{\Omega}\left(\kappa^{-2}A\nabla u\cdot \nabla \bar{v}- \alpha\beta u\bar{v}\right){\rm d}x-\langle \kappa_0^{-2}\hat{\mathscr{B}}_{\rm TE} u,v\rangle_{\Gamma_R^+}.
\end{equation*}

\subsection{Convergence of the PML problem}

Consider a Dirichlet boundary value problem of the PML equation in the PML layer
$\Omega^{\rm PML}$:
\begin{equation}\label{TE-PML-P3}
\left\{
\begin{array}{ll}
 \nabla\cdot(\kappa_0^{-2}A \nabla w)+\alpha\beta
w=0&\quad\mathrm{in}~\Omega^{\rm PML},\\
  w=0&\quad\mathrm{on}~\Gamma_R^+,\\
  w=q&\quad {\rm on}~\Gamma_{\rho}^+,\\
   (A \nabla w)\cdot \nu=0&\quad\mathrm{on}~\Gamma_g,
  \end{array}\right.
\end{equation}
where $q\in H^{1/2}_{\rm TE}(\Gamma_{\rho}^+)$.

Define $H^1_{R\rho}(\Omega^{\rm PML})=\{v\in H^1(\Omega^{\rm PML}): v=0 ~{\rm
on}~\Gamma_R^+ ~{\rm and}~\Gamma_{\rho}^+\}$. The weak formulation of
\eqref{TE-PML-P3} reads as follows: given $q\in H^{1/2}_{\rm
TE}(\Gamma_{\rho}^+)$, find $w\in H^1(\Omega^{\rm PML})$ such that $w=0~{\rm
on}~ \Gamma_R^+,w=q~{\rm on}~ \Gamma_{\rho}^+$ and
\begin{equation}\label{TE-PML-P9}
\hat{b}(w, v)=0 \quad \forall\, v\in H^1_{R\rho}(\Omega^{\rm PML}),
\end{equation}
where
\begin{equation*}
\hat{b}(u, v)=\int_R^{\rho}\int_0^{\pi}\left(\kappa_0^{-2}\left(\frac{\beta
r}{\alpha} \partial_r u \partial_r \bar{v}+\frac{\alpha}{\beta
r} \partial_\phi u\partial_\phi \bar{v}\right)-\alpha\beta
r u\bar{v}\right){\rm d}r{\rm d}\phi.
\end{equation*}
Here we also assume that the PML problem \eqref{TE-PML-P9} admits a unique weak
solution.

For any $u\in H^1(\Omega^{\rm PML})$, define
\begin{equation*}
\|u\|_{\ast,\Omega^{\rm PML}}=\left[\int_R^{\rho}\int_0^{\pi}\left(\Big(\frac{
1+\sigma\hat{\sigma}}{1+\sigma^2}\Big)r |\partial_r
u|^2+\Big(\frac{1+\sigma\hat{\sigma}}{1+\hat{\sigma}^2}\Big)\frac{1}{ r}
|\partial_\phi u|^2+ (1+\sigma\hat{\sigma})\kappa_0^2 r|u|^2\right){\rm d}r{\rm
d}\phi \right]^{1/2}.
\end{equation*}
It is easy to see that $\|\cdot\|_{\ast,\Omega^{\rm PML}}$ is an equivalent norm
on $H^1(\Omega^{\rm PML})$. By using the general theory in \cite[Chapter
5]{BA73}, there exists a positive constant $\hat{C}$ such that
\begin{equation*}
\sup_{0\neq v\in H^1_0(\Omega^{\rm
PML})}\frac{|\hat{b}(u, v)|}{\|v\|_{\ast,\Omega^{\rm
PML}}}\geq\hat{C}\|u\|_{\ast,\Omega^{\rm PML}}\quad \forall\, u\in
H^1_{R\rho}(\Omega^{\rm PML}),
\end{equation*}
The constant $\hat{C}$ depends on the domain $\Omega^{\rm PML}$ and the wave
number $\kappa_0$.

The following results concern the estimates of the solution for the boundary
value problem \eqref{TE-PML-P3} and are crucial for the convergence analysis.
The proof is essentially the same as that in \cite[Theorem 2.4]{CL05} and is
omitted for brevity.

\begin{theorem}
There exists a constant $C>0$ independent of $\kappa_0, R, \rho$, and $\sigma_0$
such that the following estimates are satisfied:
\begin{eqnarray}
\label{TE-e1}\||\alpha|^{-1}\nabla w\|_{L^2(\Omega^{\rm PML})} &\leq&
C\hat{C}^{-1}\kappa_0^{-2}(1+\kappa_0 R)|\alpha_0|\|q\|_{H^{-1/2}_{\rm
TE}(\Gamma_{\rho}^+)},\\
\label{TE-e2}\left\|\partial_r w\right\|_{H^{-1/2}_{\rm TM}(\Gamma_R^+)} &\leq&
C\hat{C}^{-1}\kappa_0^{-2}(1+\kappa_0 R)^2|\alpha_0|^2\|q\|_{H^{-1/2}_{\rm
TE}(\Gamma_{\rho}^+)},
\end{eqnarray}
where $\alpha_0=1+i\sigma_0$.
\end{theorem}

Similarly, for any function $f\in H^{1/2}_{\rm TE}(\Gamma_R^+)$, we introduce
the propagation operator $\mathscr P_{\rm TE}:H^{1/2}_{\rm
TE}(\Gamma_R^+)\rightarrow H^{1/2}_{\rm TE}(\Gamma_{\rho}^+)$ as follows:
\begin{equation*}
\mathscr P_{\rm TE}(f)=\sum_{n=0}^{\infty}
\frac{H_n^{(1)}(\kappa_0\tilde{\rho})}{H_n^{(1)} (\kappa_0R)}f_n\cos(n\phi),
\end{equation*}
where
\begin{equation*}
   f_0=\frac{1}{\pi}\int_0^{\pi}f(R,\phi){\rm d}\phi,\quad f_n=\frac{2}{\pi}\int_0^{\pi}f(R,\phi)\cos (n \phi){\rm d}\phi.
\end{equation*}
As discussed in \cite{CL05}, the operator $\mathscr P_{\rm TE}:H^{1/2}_{\rm
TE}(\Gamma_R^+)\rightarrow H^{1/2}_{\rm TE}(\Gamma_{\rho}^+)$ is well defined
and satisfies the estimate
\begin{equation}\label{TE-e3}
\|\mathscr P_{\rm TE}(f)\|_{H^{1/2}_{\rm TE}(\Gamma_{\rho}^+)}\leq
e^{-\kappa_0\Im(\tilde{\rho})\left(1-\frac{R^2}{|\tilde{\rho}|^2}\right)^{1/2}}
\|f\|_{H^{1/2}_{\rm TE}(\Gamma_R^+)}\quad \forall \, \rho\geq R.
\end{equation}

\begin{lemma}\label{TE:TBCC}
For any $f\in H^{1/2}_{\rm TE}(\Gamma_R^+)$, the following estimate holds:
\begin{equation}\label{TE-TBCC}
\|(\mathscr{B}_{\rm TE}-\hat{\mathscr{B}}_{\rm TE})f\|_{H^{-1/2}_{\rm
TE}(\Gamma_R^+)}\leq C\hat{C}^{-1}\kappa_0^{-2}(1+\kappa_0R)^2|\alpha_0|^2e^{
-\kappa_0\Im(\tilde{\rho
})\left(1-\frac{R^2}{|\tilde{\rho}|^2}\right)^{1/2}}\|f\|_{H^{1/2}_{\rm
TE}(\Gamma_R^+)}.
\end{equation}
\end{lemma}

\begin{proof}
For any $f\in H^{1/2}_{\rm TE}(\Gamma_R^+)$, we have
\begin{equation}\label{TE-TBCdiffer}
(\mathscr{B}_{\rm TE}-\hat{\mathscr{B}}_{\rm TE})f=\partial_r w|_{\Gamma_R^+},
\end{equation}
where $w\in H^1(\Omega^{\rm PML})$ satisfies
\begin{equation*}
\left\{
\begin{array}{ll}
 \nabla\cdot(\kappa_0^{-2}A \nabla w)+\alpha\beta
w=0&\quad\mathrm{in}~\Omega^{\rm PML},\\
  w=0&\quad\mathrm{on}~\Gamma_R^+,\\
   w=\mathscr P_{\rm TE}(f)&\quad {\rm on}~\Gamma_{\rho}^+,\\
   (A \nabla w)\cdot\nu=0&\quad\mathrm{on}~\Gamma_g.\\
  \end{array}\right.
\end{equation*}
It follows from \eqref{TE-e2}--\eqref{TE-e3} that
\begin{eqnarray*}
  \left\|\partial_r w\right\|_{H^{-1/2}_{\rm TM}(\Gamma_R^+)} &\leq&
C\hat{C}^{-1}\kappa_0^{-2}(1+\kappa_0R)^2|\alpha_0|^2\|\mathscr
P_{\rm TE}(f)\|_{H^{1/2}_{ \rm
TM}(\Gamma_{\rho}^+)}\\
&\leq& C\hat{C}^{-1}\kappa_0^{-2}(1+\kappa_0R)^2|\alpha_0|^2e^{
-\kappa_0\Im(\tilde{\rho
})\left(1-\frac{R^2}{|\tilde{\rho}|^2}\right)^{1/2}}\|f\|_{H^{1/2}_{\rm
TM}(\Gamma_R^+)},
\end{eqnarray*}
which completes the proof.
\end{proof}

Below is the main result of this subsection.

\begin{theorem}
For sufficiently large $\sigma_0>0$, the PML problem \eqref{TE-PML-P2} has a
unique solution $u^{\rm PML}\in H^1(\Omega_{\rho})$. Moreover, the
following estimate holds:
\begin{equation*}
\|u-u^{\rm PML}\|_{H^1(\Omega)}\leq
C\hat{C}^{-1}(\kappa_0^{-2}+\kappa_0^{-1}R)^2|\alpha_0|^2e^{-\kappa_0\Im(\tilde{
\rho})(1-\frac{R^2}{|\tilde{\rho}|^2})^{1/2}}\|u^{\rm PML}-u^{\rm
ref}\|_{H^{1/2}_{\rm TM}(\Gamma_R^+)}.
\end{equation*}
\end{theorem}

\begin{proof}
By \eqref{TE-VF} and \eqref{TE-PML-VF}, for any $\varphi \in H^1(\Omega)$, we
have
\begin{eqnarray*}
  a_{\rm TE}(u-u^{\rm PML},\varphi) &=&a_{\rm TE}(u,\varphi)-a_{\rm TE}(u^{\rm
PML},\varphi)  \\
   &=&\langle \kappa_0^{-2}g,\varphi\rangle_{\Gamma_R^+}-a_{\rm TE}(u^{\rm
PML},\varphi)  \\
   &=& \langle\kappa_0^{-2}(g-\hat{g}),\varphi\rangle_{\Gamma_R^+}+\langle
\kappa_0^{-2}\hat{g},\varphi\rangle_{\Gamma_R^+}-a(u^{\rm PML},\varphi) \\
  &=& \langle\kappa_0^{-2} (\hat{\mathscr{B}}_{\rm TE}-\mathscr{B}_{\rm
TE})u^{\rm ref},\varphi\rangle_{\Gamma_R^+}+\hat{a}_{\rm TE}(u^{\rm
PML},\varphi)-a_{\rm TE}(u^{\rm PML},\varphi) \\
  &=&\langle\kappa_0^{-2} (\mathscr{B}_{\rm TE}-\hat{\mathscr{B}}_{\rm
TE})(u^{\rm PML}-u^{\rm ref}),\varphi\rangle_{\Gamma_R^+},
\end{eqnarray*}
which implies the desired estimate by using Lemma \ref{TE:TBCC} and
\eqref{TE-is1}.
\end{proof}

\subsection{Finite element approximation}

Let $b(\cdot, \cdot):H^1(\Omega_{\rho})\times H^1(\Omega_{\rho})\rightarrow
\mathbb{C}$ be the sesquilinear form given by
 \begin{equation*}
  b(u, v)=\int_{\Omega_{\rho}}(\kappa^{-2}A\nabla u\cdot
\nabla\bar{v}-\alpha\beta u\bar{v}){\rm d}x.
\end{equation*}
Define $ H^1_{\rho}(\Omega_{\rho})=\{u\in H^1(\Omega_{\rho}): u=0~{\rm on}~
\Gamma_{\rho}^+\}$. The the weak formulation of \eqref{TE-PML-P1} is to find
$u^{\rm PML}\in H^1(\Omega_{\rho})$ and $u^{\rm PML}=u^{\rm ref} ~{\rm
on}~\Gamma_{\rho}^+ $ such that
\begin{equation}\label{TE-PML-P5}
b(u^{\rm PML}, v)=-\int_{\Omega_{\rho}}G\bar{v}{\rm d}x\quad \forall\,
v\in H^1_{\rho}(\Omega_{\rho}).
\end{equation}
Let $V_h$ be a conforming finite element space of $H^1(\Omega_\rho)$, i.e.,
\begin{equation*}
  V_h=\{v_h\in C(\bar{\Omega}_{\rho}):v_h|_K\in P_m(K), \forall K\in
\mathcal{M}_h\}.
\end{equation*}
Denote
\begin{equation*}
  V_{\rho, h}=\{v_h\in V_h: v_h=0 ~{\rm on} ~\Gamma_{\rho}^+\}.
\end{equation*}
The finite element approximation to the variational problem \eqref{TE-PML-P5} is to find $u_h\in V_h$ with $u_h=u^{\rm ref} ~{\rm on}~ \Gamma_{\rho}^+$ such that
\begin{equation}\label{TE-PML-P7}
b(u_h, v_h)=-\int_{\Omega_{\rho}}G\bar{v}_h{\rm d}x \quad \forall\,
v_h\in V_{\rho, h}.
\end{equation}

\subsection{A posteriori error analysis}

For any $K\in \mathcal{M}_h$, we introduce the residual
\begin{equation*}
  R_K(u):=\nabla\cdot(\kappa^{-2}A\nabla u|_K)+\alpha\beta u|_K.
\end{equation*}
Let $\mathcal{B}_h$ denote the set of all the edges that do not lie on $\partial \Gamma_{\rho}^+$. For any interior edge $e$, which is the common side of triangular elements $K_1,K_2\in \mathcal{M}_h$, we define the jump residual across $e$ as
\begin{equation*}
 J_e:=-(\kappa^{-2}A\nabla u_h|_{K_1}\cdot \nu_1+\kappa^{-2}A\nabla
u_h|_{K_2}\cdot \nu_2),
\end{equation*}
where $\nu_j$ is the unit outward normal vector on the boundary of $K_j,j=1,2$. If $e=\partial K\cap(\Gamma_g\cup S) $ for some $K\in \mathcal{M}_h$, then define the jump residual
\begin{equation*}
J_e:=2(\kappa^{-2}A\nabla u_h|_K\cdot \nu).
\end{equation*}
Let
\begin{equation*}
\tilde{R}_K:=
\left\{
\begin{array}{ll}
 R_K(u_h)&\quad{\rm if}~ K\in \mathcal{M}_h\cap \Omega,\\
  R_K(u_h-u^{\rm ref})&\quad {\rm if}~ K\in \mathcal{M}_h\cap \Omega^{\rm PML}.
  \end{array}\right.
\end{equation*}
For any triangle $K\in \mathcal{M}_h$, denote by $\eta_K$ the local error estimator as follows:
\begin{equation*}
  \eta_K=\max_{x\in
\tilde{K}}w(x)\Big(\|h_K\tilde{R}_K\|_{L^2(K)}^2+\frac{1}{2}\sum_{e\in
K}\|h_e^{1/2}J_e\|_{L^2(e)}^2\Big)^{1/2},
\end{equation*}
where
\begin{equation*}
w(x)=
\left\{
\begin{array}{ll}
 1&\quad{\rm if}~ x\in \bar{\Omega},\\
|\frac{\alpha}{\alpha_0}|e^{-\kappa_0\Im{\tilde{r}}\left(1-\frac{r^2}{
|\tilde{r}|^2}\right)^{1/2}}&\quad {\rm if}~ x\in \Omega^{\rm PML}.
  \end{array}\right.
\end{equation*}

For any $\varphi\in H^1(\Omega_R)$, let $\tilde{\varphi}$ be its extension in $\Omega^{\rm PML}$ such that
\begin{equation}\label{TE-PML-P8}
\left\{
\begin{array}{ll}
 \nabla\cdot(\kappa_0^{-2}\bar {A} \nabla
\tilde{\varphi})+\overline{\alpha\beta}\tilde{\varphi}
=0&\quad\mathrm{in}~\Omega^{\rm PML},\\
  \tilde{\varphi}=\varphi&\quad\mathrm{on}~\Gamma_R^+,\\
   \tilde{\varphi}=0&\quad {\rm on}~\Gamma_{\rho}^+,\\
   (\bar{A}\nabla\tilde{\varphi})\cdot \nu=0&\quad {\rm on}~\Gamma_g.
  \end{array}\right.
\end{equation}

The proofs of Lemmas \ref{TEL:L4} and \ref{TEL:L1} are essentially the same as
those in  \cite[Lemmas 4.1 and 4.4]{CL05}, while the proof of Lemma
\ref{TEL:L2} is also similar to Lemma \ref{TML:L2}.

\begin{lemma}\label{TEL:L4}
For any $\varphi\in H^1(\Omega)$, let $\tilde{\varphi}$ be its extension in
$H^1(\Omega^{\rm PML})$ according to \eqref{TE-PML-P8}. Then
there exists a constant $C>0$ independent of $\kappa_0,R,\rho$ and $\sigma_0$
such that
\begin{equation*}
\||\alpha|^{-1}\gamma\nabla\tilde{\varphi}\|_{L^2(\Omega^{\rm PML})}\leq C\hat{C}^{-1}\kappa_0^{-2}(1+\kappa_0 R)|\alpha_0|\|\varphi\|_{H^{1/2}(\Gamma_R^+)},
\end{equation*}
where $\gamma(r)=e^{\kappa_0\Im{\tilde{r}}\left(1-\frac{r^2}{|\tilde{r}|^2}\right)^{1/2}}$.
\end{lemma}

\begin{lemma}\label{TEL:L1}
For any $\varphi,\psi\in H^1(\Omega^{\rm PML})$, we have
\begin{equation*}
  \langle\hat{\mathscr{B}}_{\rm TE}\varphi,\psi\rangle=\langle\hat{\mathscr{B}}_{\rm TE}\bar{\psi},\bar{\varphi}\rangle
\end{equation*}
\end{lemma}

\begin{lemma}\label{TEL:L2}
For any $\varphi\in H^1(\Omega)$, let $\tilde\varphi$ be its extension in
$H^1(\Omega_{\rho})$ according to \eqref{TE-PML-P8}. For $\xi\in
H_{0}^1(\Omega_{\rho})$, the following identity holds:
\begin{equation*}
\int_{\Omega^{\rm PML}}\left(\kappa_0^{-2}A\nabla\xi\cdot \nabla\bar{\tilde{\varphi}}-\alpha\beta\xi\bar{\tilde{\varphi}} \right){\rm d}x=-\langle\kappa_0^{-2}\hat{\mathscr{B}}_{\rm TE}\xi,\varphi\rangle_{\Gamma_R^+}.
\end{equation*}
\end{lemma}

The following result present the error representation formula for the TE
polarization.

\begin{lemma}[error representation formula]\label{TEL:L3}
For any $\varphi\in H^1(\Omega)$, let $\tilde\varphi$ be its extension in
$H^1(\Omega_{\rho})$ according to \eqref{TE-PML-P8}. The for any $\varphi_h\in
V_h$, the following identity holds:
\begin{eqnarray*}
a_{\rm TE}(u-u_h,\varphi)&=&\langle\kappa^{-2} \mathscr{B}_{\rm TE}(u_h-u^{\rm
ref})-\kappa^{-2} \hat{\mathscr{B}}_{\rm TE}(u_h-u^{\rm ref}),
\varphi\rangle_{\Gamma_R^+}-b(u_h,\varphi-\varphi_h)\\
&&-\int_{\Omega^{\rm PML}}\big(  \nabla\cdot(\kappa_0^{-2}A \nabla u^{\rm
ref})+\alpha\beta u^{\rm ref}\big)(\bar{\tilde{\varphi}}-\bar{\varphi}_h){\rm
d}x.
\end{eqnarray*}
\end{lemma}

\begin{proof}
It follows from \eqref{TE-VF} that
\begin{eqnarray}
\nonumber a_{\rm TE}(u-u_h,\varphi)&=&a_{\rm TE}(u,\varphi)-a_{\rm TE}(u_h,\varphi)  \\
\nonumber  &=&\langle\kappa^{-2} g, \varphi\rangle_{\Gamma_R^+}-b(u_h,\varphi-\varphi_h)\\\label{TE-eq5}
  &&+b(u_h,\varphi)-b(u_h,\varphi_h)-a_{\rm TE}(u_h,\varphi).
\end{eqnarray}
By the definition of the sesquilinear form $b$, we have
\begin{eqnarray*}
  b(u_h,\varphi) &=&\int_{\Omega_{\rho}}\left(\kappa^{-2}A \nabla u_h  \cdot\nabla \bar{\varphi}-\alpha\beta u_h\bar{\varphi}\right){\rm d}x  \\
   &=& \int_{\Omega}\left(\kappa^{-2}A \nabla u_h  \cdot\nabla \bar{\varphi}-\alpha\beta u_h\bar{\varphi}\right){\rm d}x+\int_{\Omega^{\rm PML}}\left(\kappa_0^{-2}A \nabla u_h  \cdot\nabla \bar{\tilde{\varphi}}-\alpha\beta u_h\bar{\tilde{\varphi}}\right){\rm d}x.
\end{eqnarray*}
We also get from the definition of the sesquilinear form $a_{\rm TE}$ that
\begin{eqnarray*}
  a_{\rm TE}(u_h,\varphi) &=&\int_{\Omega}\left(\kappa^{-2}A \nabla u_h  \cdot\nabla \bar{\varphi}-\alpha\beta u_h\bar{\varphi}\right){\rm d}x- \langle\kappa_0^{-2} \mathscr{B}_{\rm TE}u_h, \varphi\rangle_{\Gamma_R^+}.
\end{eqnarray*}
Using \eqref{TE-PML-P5} and the integration by parts yields
\begin{eqnarray*}
  b(u_h,\varphi_h) &=&- \int_{\Omega^{\rm PML}}G\bar{\varphi}_h{\rm d}x \\
   &=&-\int_{\Omega^{\rm PML}}\big(  \nabla\cdot(\kappa_0^{-2}A \nabla u^{\rm
ref})+\alpha\beta u^{\rm ref}\big)\bar{\varphi}_h{\rm d}x\\
   &=&\int_{\Omega^{\rm PML}}\big(\nabla\cdot(\kappa_0^{-2}A \nabla u^{\rm
ref})+\alpha\beta u^{\rm ref}\big)(\bar{\tilde{\varphi}}-\bar{\varphi}_h){\rm
d}x+\int_{\Omega^{\rm PML}}\big(\kappa_0^{-2} A \nabla u^{\rm ref}\cdot\nabla
\bar{\tilde{\varphi}} -\alpha\beta u^{\rm ref}\bar{\tilde{\varphi}}\big){\rm
d}x\\
   &&+\int_{\Gamma_R^+}\kappa_0^{-2}\partial_\nu u^{\rm
ref}\bar{\varphi}{\rm d}s.
\end{eqnarray*}
Combining the above equations leads to
\begin{eqnarray*}
   &&b(u_h,\varphi)-b(u_h,\varphi_h)-a_{\rm TE}(u_h,\varphi)  \\
   &=&-\int_{\Omega^{\rm PML}}\big(  \nabla\cdot(\kappa_0^{-2}A \nabla u^{\rm
ref})+\alpha\beta u^{\rm ref}\big)(\bar{\tilde{\varphi}}-\bar{\varphi}_h){\rm
d}x\\
   &&+\int_{\Omega^{\rm PML}}\big(\kappa_0^{-2}A \nabla (u_h-u^{\rm ref})
\cdot\nabla \bar{\tilde{\varphi}}-\alpha\beta (u_h-u^{\rm
ref})\bar{\tilde{\varphi}}\big){\rm d}x\\
   &&-\int_{\Gamma_R^+}\kappa_0^{-2}\partial _{\nu}u^{\rm ref}\bar{\varphi}{\rm d}s+ \langle\kappa_0^{-2} \mathscr{B}_{\rm TE}u_h, \varphi\rangle_{\Gamma_R^+}.
\end{eqnarray*}
By Lemma \ref{TEL:L2}, we have
\begin{eqnarray}
  \nonumber &&b(u_h,\varphi)-b(u_h,\varphi_h)-a_{\rm TE}(u_h,\varphi)  \\
   \nonumber&=&-\int_{\Omega^{\rm PML}}\big(  \nabla\cdot(\kappa_0^{-2}A \nabla
u^{\rm ref})+\alpha\beta u^{\rm
ref}\big)(\bar{\tilde{\varphi}}-\bar{\varphi}_h){\rm d}x\\
 &&+\langle\kappa_0^{-2} (\mathscr{B}_{\rm TE}-
\hat{\mathscr{B}}_{\rm TE})u_h,
\varphi\rangle_{\Gamma_R^+}
  -\langle\kappa_0^{-2}(\partial_{\nu} u^{\rm ref}-
\hat{\mathscr{B}}_{\rm TE}u^{\rm ref}),
\varphi\rangle_{\Gamma_R^+}.\label{TE-eq6}
\end{eqnarray}
Substituting \eqref{TE-eq6} into \eqref{TE-eq5}, we obtain
\begin{eqnarray*}
a_{\rm TE}(u-u_h,\varphi)&=&\langle\kappa_0^{-2} \mathscr{B}_{\rm TE}(u_h-u^{\rm ref})-\kappa_0^{-2} \hat{\mathscr{B}}_{\rm TE}(u_h-u^{\rm ref}), \varphi\rangle_{\Gamma_R^+}-b(u_h,\varphi-\varphi_h)\\
&&-\int_{\Omega^{\rm PML}}\big(  \nabla\cdot(\kappa_0^{-2}A \nabla u^{\rm
ref})+\alpha\beta u^{\rm ref}\big)(\bar{\tilde{\varphi}}-\bar{\varphi}_h){\rm
d}x,
\end{eqnarray*}
which completes the proof.
\end{proof}

The following theorem presents the a posteriori error estimate and is the main
result for the TE polarization.

\begin{theorem} \label{TE:main}
Let $u$ and $u_h$ be the solution of \eqref{TE-VF} and \eqref{TE-PML-P7},
respectively. Then there exists a positive constant $C$ depending only on the
minimum angle of the mesh $\mathcal{M}_h$ such that the following the a
posteriori error estimate holds:
\begin{eqnarray*}
\|u-u_h\|_{H^1(\Omega)} &\leq& C\hat{C}^{-1}\kappa_0^{-2}(1+\kappa
R)\bigg(\sum_{K\in\mathcal{M}_h}\eta_K^2\bigg)^{1/2}\\
&&+C\hat{C}^{-1}\kappa_0^{-2}(1+\kappa_0R)^2|\alpha_0|^2e^{-\kappa_0\Im{(\tilde{
\rho})}(1-\frac{R^2}{|\tilde{\rho}|^2})^{1/2}}\|u_h-u^{\rm ref}\|_{H^{1/2}_{\rm
TE}(\Gamma_R^+)} .
\end{eqnarray*}
\end{theorem}

\begin{proof}
Taking $\varphi_h=\Pi_h \varphi$ and using Lemma \ref{TEL:L3}, we have
\begin{eqnarray*}
a_{\rm TE}(u-u_h,\varphi)&=&\langle\kappa_0^{-2} \mathscr{B}_{\rm TE}(u_h-u^{\rm
ref})-\kappa_0^{-2} \hat{\mathscr{B}}_{\rm TE}(u_h-u^{\rm ref}),
\varphi\rangle_{\Gamma_R^+}-b(u_h,\varphi-\varphi_h)\\
&&-\int_{\Omega^{\rm PML}}\big(  \nabla\cdot(\kappa_0^{-2}A \nabla
u^{\rm ref})+\alpha\beta u^{\rm
ref}\big)(\bar{\tilde{\varphi}}-\bar{\varphi}_h){\rm d}x\\
 &:=&{\rm I_1}+{\rm I}_2+{\rm I}_3.
\end{eqnarray*}
By Lemma \ref{TE-TBCC}, we get
\begin{eqnarray*}
|{\rm I}_1| &=&\left| \langle\kappa_0^{-2} \mathscr{B}_{\rm TE}(u_h-u^{\rm ref})
-\kappa_0^{-2}\mathscr{B}_{\rm TE}(u_h-u^{\rm
ref}),\varphi\rangle_{\Gamma_R^+}\right| \\
&\leq&C\hat{C}^{-1}\kappa_0^{-2}(1+\kappa_0R)^2|\alpha_0|^2e^{-\kappa_0\Im{
(\tilde{\rho})}(1-\frac{R^2}{|\tilde{\rho}|^2})^{1/2}}\|u_h-u^{\rm
ref}\|_{H^{1/2}_{\rm TE}(\Gamma_R^+)}\|\varphi\|_{H^{1/2}_{\rm TE}(\Gamma_R^+)}.
\end{eqnarray*}
It is easy to see that
\begin{eqnarray*}
{\rm I}_2+{\rm I}_3 &=&\sum_{K\in\mathcal{M}_h\cap \Omega}\bigg(\int_K
R_K(u_h)(\bar{\varphi}-\Pi_h \bar{\varphi}){\rm d}x+\sum_{e\in\partial
K}\frac{1}{2}\int_e J_e (\bar{\varphi}-\Pi_h \bar{\varphi}){\rm d}s\bigg)\\
  &+&  \sum_{K\in\mathcal{M}_h\cap \Omega^{\rm PML}}\bigg(\int_K R_K(u_h-u^{\rm
ref})(\bar{\varphi}-\Pi_h \bar{\varphi}){\rm d}x+\sum_{e\in\partial K\cap
\mathcal{B}_h}\frac{1}{2}\int_e J_e (\bar{\varphi}-\Pi_h \bar{\varphi}){\rm
d}s\bigg)
\end{eqnarray*}
It follows from the Cauchy--Schwarz inequality, the interpolation estimates, and
lemma \ref{TEL:L4} that
\begin{eqnarray*}
|{\rm I}_2+{\rm I}_3|&\leq& C\sum_{K\in\mathcal{M}_h}\bigg(\|h_K
\tilde{R}_K\|^2_{L^2(K)}+\frac{1}{2}\sum_{e\in\partial K\cap
\mathcal{B}_h}\|h_e^{1/2}J_e\|^2_{L^2(e)} \bigg)^{1/2}\|\nabla
\varphi\|_{L^2(\tilde{K})}\\
&\leq& C \sum_{K\in\mathcal{M}_h} \eta_K\|w^{-1}\nabla
\varphi\|_{L^2(\tilde{K})}\\
&\leq&C\hat{C}^{-1}\kappa_0^{-2}(1+\kappa_0R)\bigg(\sum_{K\in\mathcal{M}_h}
\eta_K^2\bigg)^{1/2}\|\varphi\|_{H^{1/2}(\Gamma_R^+)}.
\end{eqnarray*}
By the inf-sup condition \eqref{TE-is1}, we obtain
\begin{eqnarray*}
\|u-u_h\|_{H^1(\Omega)} &\leq& C\sup_{0\neq\varphi\in
H^1_0(\Omega)}\frac{|a_{\rm TE}(u-u_h,\varphi)}{\|\varphi\|_{H^1(\Omega)}} \\
&\leq& C\hat{C}^{-1}\kappa_0^{-2}(1+\kappa
R)\bigg(\sum_{K\in\mathcal{M}_h}\eta_K^2\bigg)^{1/2}\\
&&+C\hat{C}^{-1}\kappa_0^{-2}(1+\kappa_0R)^2|\alpha_0|^2e^{-\kappa_0\Im{(\tilde{
\rho})}(1-\frac{R^2}{|\tilde{\rho}|^2})^{1/2}}\|u_h-u^{\rm ref}\|_{H^{1/2}_{\rm
TE}(\Gamma_R^+)},
\end{eqnarray*}
which completes the proof.
\end{proof}

As can be seen from the Theorem \ref{TE:main}, the a posteriori error
estimate also consists of two parts: the finite element approximation error
and the PML error which decays exponentially with respect to the PML
parameters. In practice, we may choose the PML parameters appropriately such
that the PML error is negligible compared with the finite element approximation
error. The algorithm of the adaptive finite element PML method for the TE case
is similar to that of the method for the TM polarized open cavity scattering
problem described in Table  \ref{Tab1}.

\section{Numerical experiments}

In this section, we present some numerical examples to demonstrate the
performance of the adaptive finite element PML method. The method is
validated and compared with the adaptive finite element method with the
transparent boundary condition (TBC) which is proposed in \cite{YBL20}. In the
following examples, the PML parameters are $\rho=3R$,  $\sigma_0=20$ and $m=2$.

The physical quantity of interest associated with the cavity scattering is the
radar cross section (RCS), which measures the detectability of a target by a
radar system \cite{j-02}. When the incident angle and the observation angle
are the same, the RCS is called the backscatter RCS. The specific formulas can
be found in \cite{YBL20} for the backscatter RCS on both polarized wave
fields.

\subsection{Example 1}

We consider a benchmark example for the TM polarized wave fields \cite{j-02}.
The cavity has a rectangular shape with width $\lambda$ and depth $0.25
\lambda$.
Figure \ref{TM_Example1} shows the geometry of the cavity and the PML setting.
The wavenumber in the free space is $\kappa_0=32\pi$ and wavelength
$\lambda=2\pi/\kappa_0=1/16$. The PML layer is a semi-annulus region and
is imposed above the cavity with $R=1/2\lambda$. Two cases are considered in
this example: an empty cavity with no fillings inside of the cavity and a cavity
filled by a lossy medium with the electric permittivity $\epsilon=4+{\rm i}$ and
the magnetic permeability $\mu=1$. First, we compute the backscatter
RCS by using the adaptive finite element PML method and TBC method. For both
methods, the adaptive mesh refinements are stopped once the total number of
nodal points are over 15000. The backscatter RCS is shown as red solid lines and
blue circles in Figure \ref{TM_Example1} for the adaptive PML method and
adaptive TBC method, respectively. It is clear to note that the results obtained
by both methods are consistent with each other. Using the incident angle
$\theta=\pi/4$ as a representative example in case 1, we present the
adaptively refined mesh after 3 iterations with a total number of nodal points
1259 in Figure \ref{TM_Example1_2}. As expected, the mesh is refined
near the two corners of the cavity and keep relatively coarse near the outer
boundary of the PML layer, since the solution has singularity around the two
L-shaped corners and is smooth and flat in the PML region, particularly in the
part which is close to the outer boundary of the PML layer. The a posteriori
error estimates for case 1 at incident angle $\theta=\pi/4$ are plotted in
Figure \ref{TM_Example1_2} to show the convergence rate of the method. It
indicates that the meshes and the associated numerical complexity are
quasi-optimal, i.e., $\epsilon_h=\mathcal O({\rm DoF}_h^{-1/2})$ holds
asymptotically, where ${\rm DoF}_h$ is the degree of freedom or the total number
of nodal points for the mesh $\mathcal{M}_h$. As a comparison, the a posteriori
error estimates are also plotted for the adaptive TBC method in the red dashed
line. Clearly, the method also preserves the quasi-optimality. It can be
observed that the TBC method gives a smaller error than the PML method does for
the same number of nodal points. There are two reasons: the TBC method does not
require an artificial absorbing layer to enclose the physical domain which may
reduce the size of the computational domain; the a posteriori error estimates
may not be sharp for both methods as the lower bounds are not given. But the
PML method is simpler than the TBC method from the implementation point of
view. The PML method only involves the local Dirichlet boundary condition while
the TBC method has to handle the nonlocal TBC. Figure \ref{TM_Example1_3}
plots the ratio of ${\rm DoF}_h$ bewteen the physical domain and the
whole computational domain. It shows that the physical domain asymptotically
accounts for $70\%$ of the total number of nodal points which illustrates that
most of the nodal points are concentrated inside the physical domain and the a
posteriori error estimate is effective for the PML method.

\begin{figure}
\centering
\includegraphics[width=0.45\textwidth]{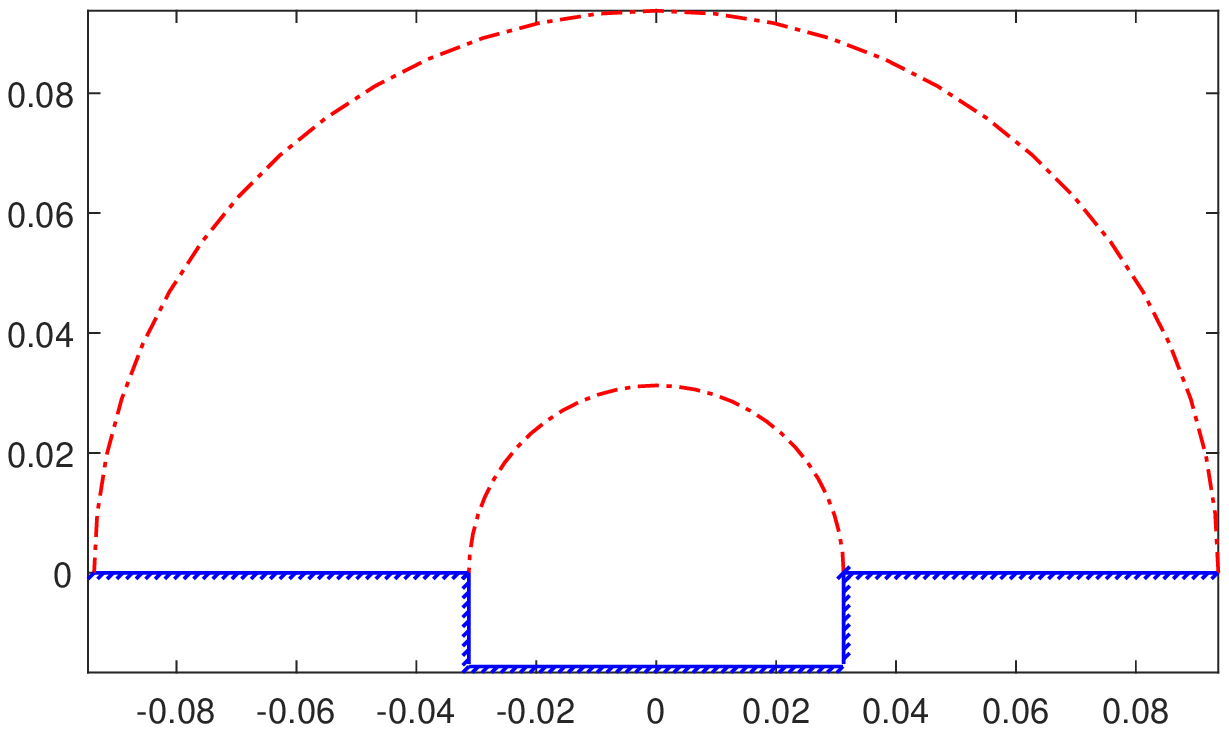}\quad
\includegraphics[width=0.45\textwidth]{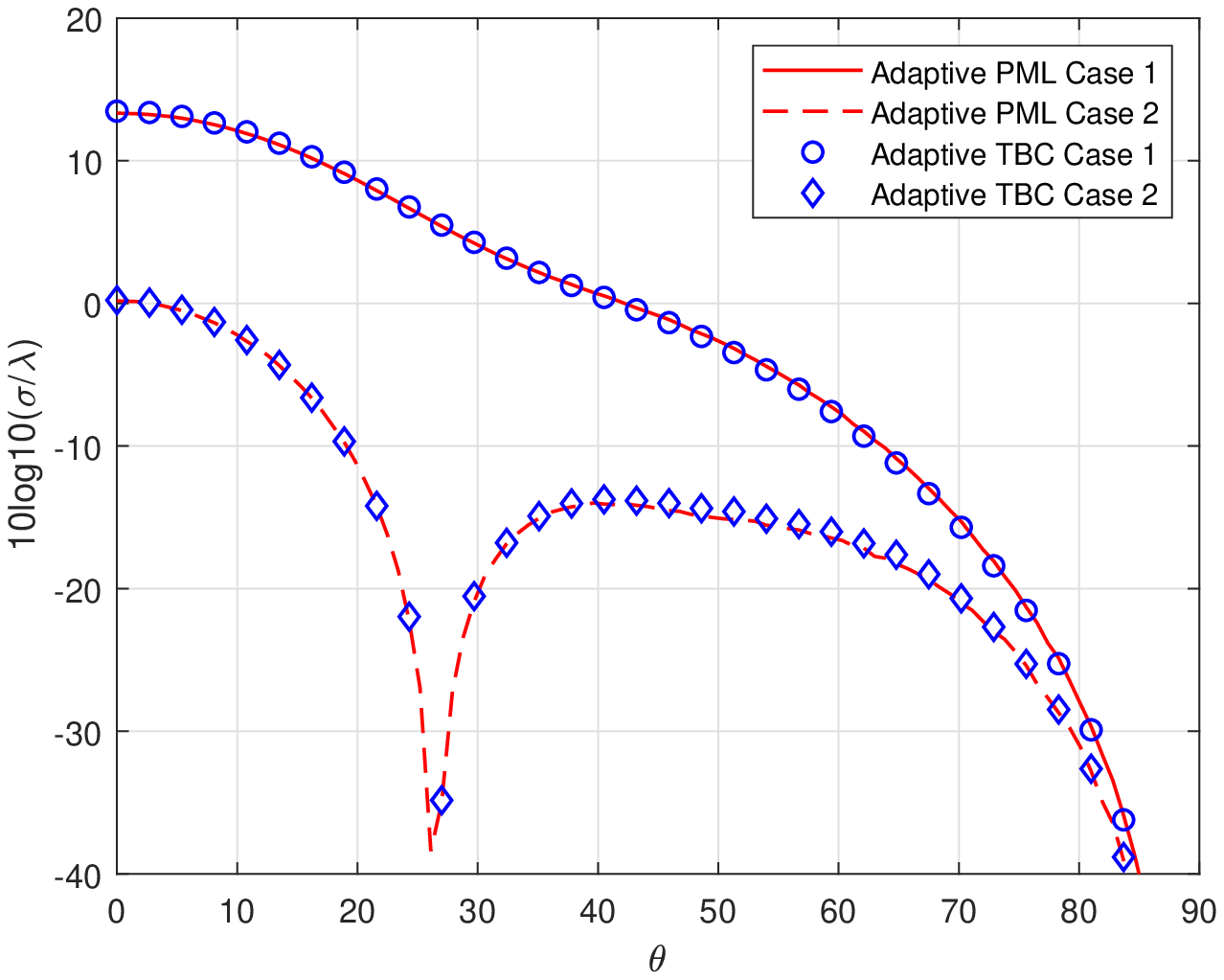}
\caption{Example 1: (left) the cavity geometry; (right) the backscatter RCS
for both cases by using the adaptive PML method and the adaptive TBC method.}
\label{TM_Example1}
\end{figure}

\begin{figure}
\centering
\includegraphics[width=0.45\textwidth]{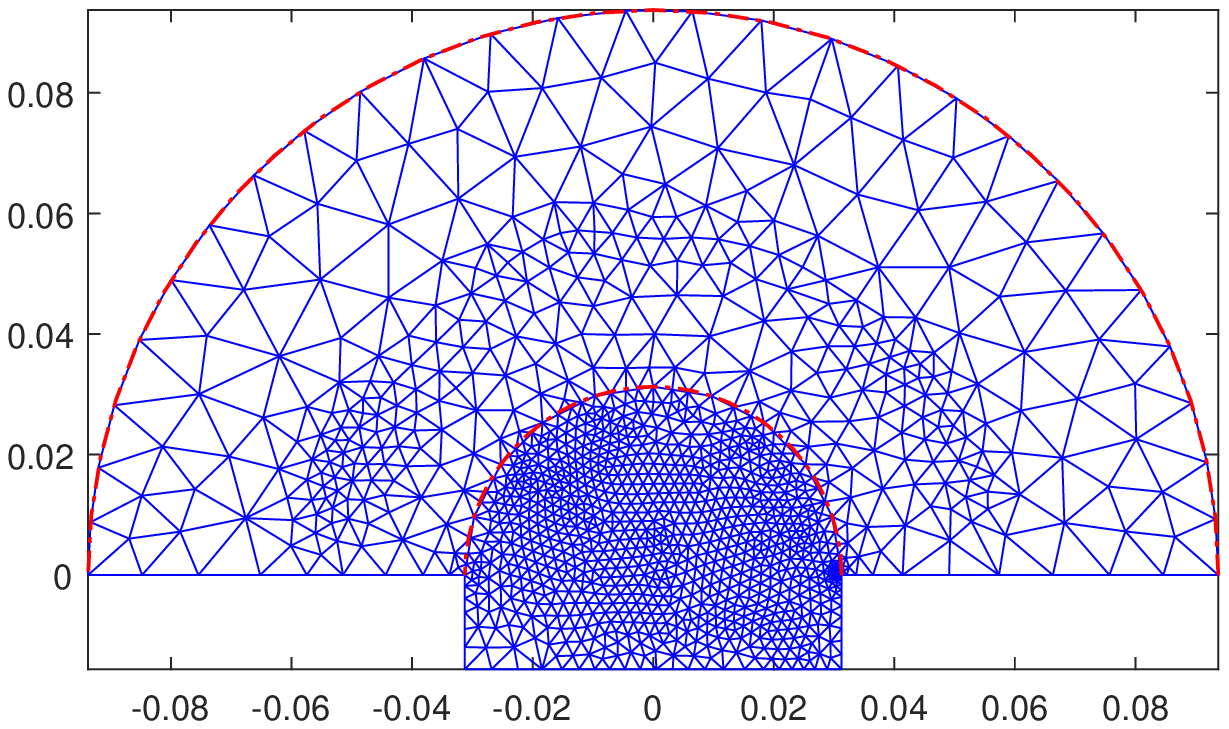}\quad
\includegraphics[width=0.45\textwidth]{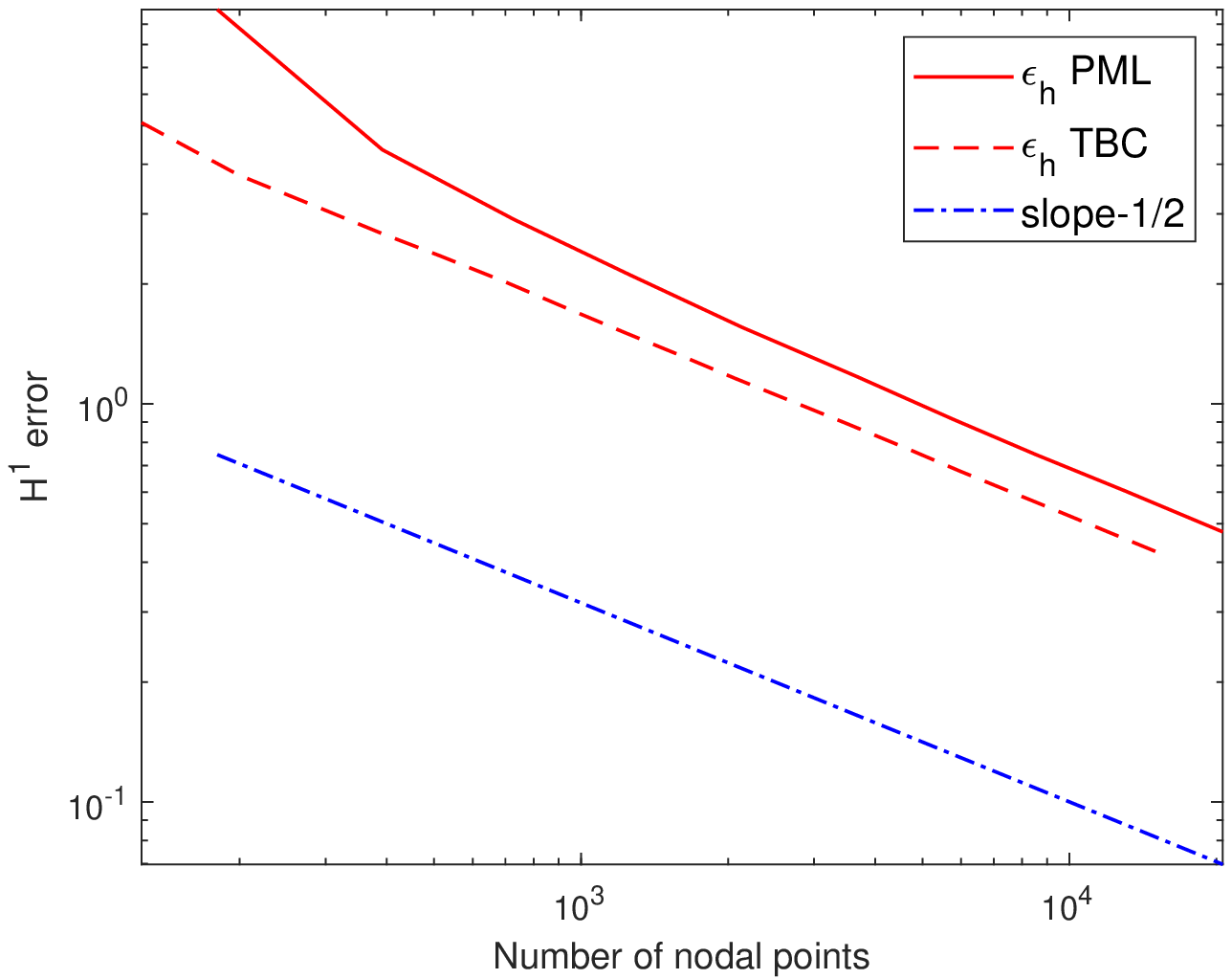}
\caption{Example 1: (left) the adaptive mesh after 3 iterations with a total
number of nodal points $1259$; (right) the quasi-optimality of the a posteriori
error estimates for both of the adaptive PML and TBC methods.}
\label{TM_Example1_2}
\end{figure}

\begin{figure}
\centering
\includegraphics[width=0.45\textwidth]{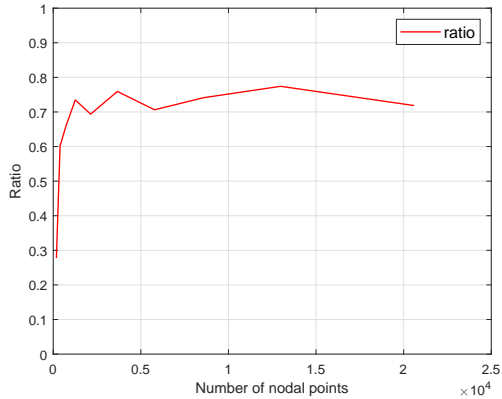}
\caption{Example 1: the ratio of ${\rm DoF}_h$ between the physical domain and
the whole computational domain.}
\label{TM_Example1_3}
\end{figure}

\subsection{Example 2}

In this example, we also consider the TM polarization. The backscatter RCS for
a coated rectangular cavity with width $2.4 \lambda$ and depth $1.6 \lambda$ is
computed. The each vertical side of the cavity wall is coated with a thin layer
of some absorbing material. Figure \ref{TM_Example2} illustrates the geometry of
the cavity and PML setting. The coating on both sides has thickness $0.024
\lambda$ and is made of a homogeneous absorbing material with  a relative
permittivity $\epsilon_{\rm r}=12+0.144 {\,\rm i}$ and a relative permeability
$\mu_{\rm r}=1.74+3.306{\,\rm i}$. This example has a multi-scale feature and
is an interesting benchmark example to test the adaptive method. We take the
same stopping rule as that in Example 1: the adaptive method is stopped once the
number of nodal points is over 15000. Figure \ref{TM_Example2} plots the
backscatter RCS by using the adaptive PML and TBC methods, where
the red solid line stands for the results of the PML method and the blue circles
stand for the results of the TBC method. Clearly, these two methods are
consistent with each other. Using a representative example of incident angle
$\theta=\pi/4$, we present the refined mesh after 2 iterations with 1263 ${\rm
DoF}_h$ and the a posteriori error estimates in Figure \ref{TM_Example2_2}. It
is clear to note that the method can capture the behavior of the numerical
solution in the two thin absorbing layers and displays the quasi-optimality
between the meshes and the associated numerical complexity, i.e.,
$\epsilon_h=\mathcal O({\rm DoF}_h^{-1/2})$ holds asymptotically. As a
comparison, we also show the a posteriori error estimates of the adaptive TBC
method in the red dashed line. The quasi-optimality is also observed for the
adaptive TBC method. Figure \ref{TM_Example2_3} shows the ratio of
${\rm DoF}_h$ between the physical domain and the whole computational domain.
Again, we see that most nodal points are concentrated in the physical domain,r
which illustrates the effectiveness of the adaptivity.

\begin{figure}
\centering
\includegraphics[width=0.45\textwidth]{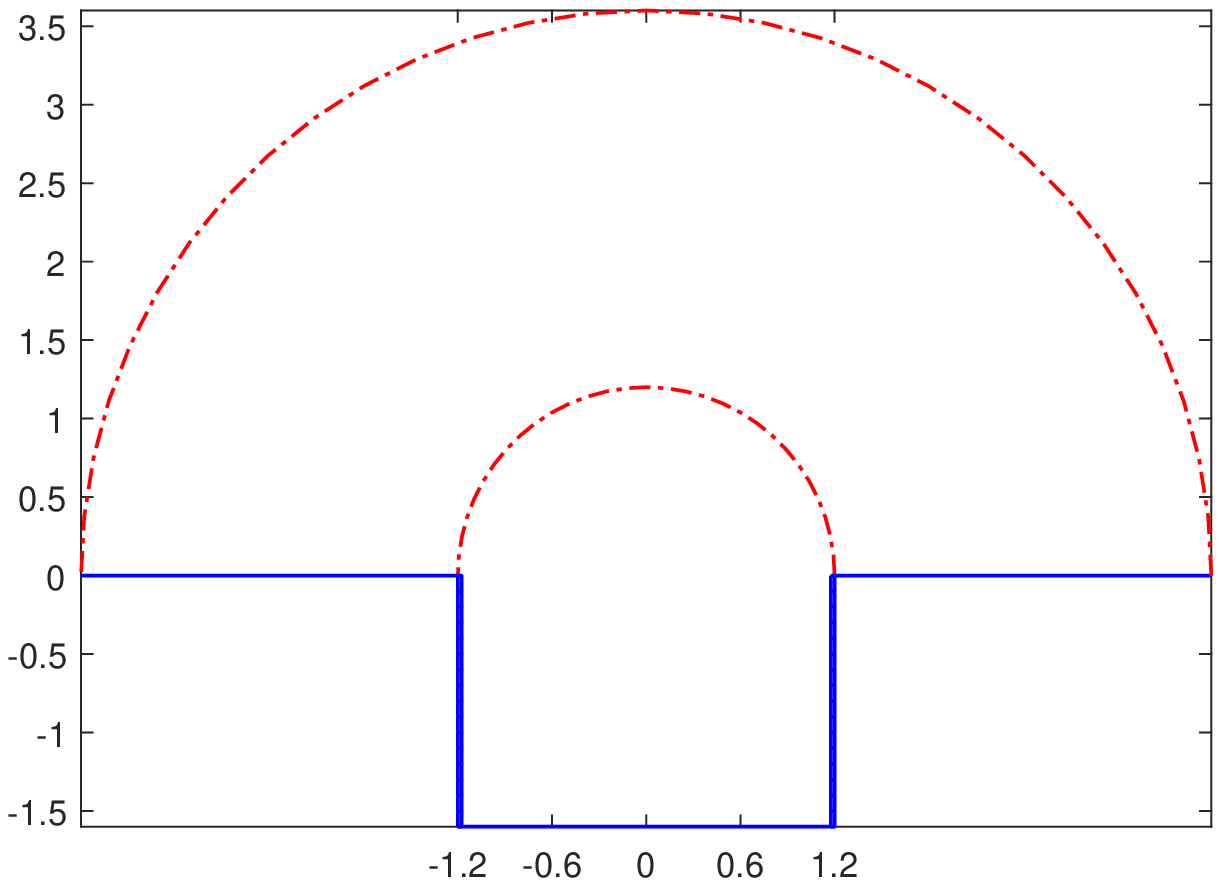}\quad
\includegraphics[width=0.45\textwidth]{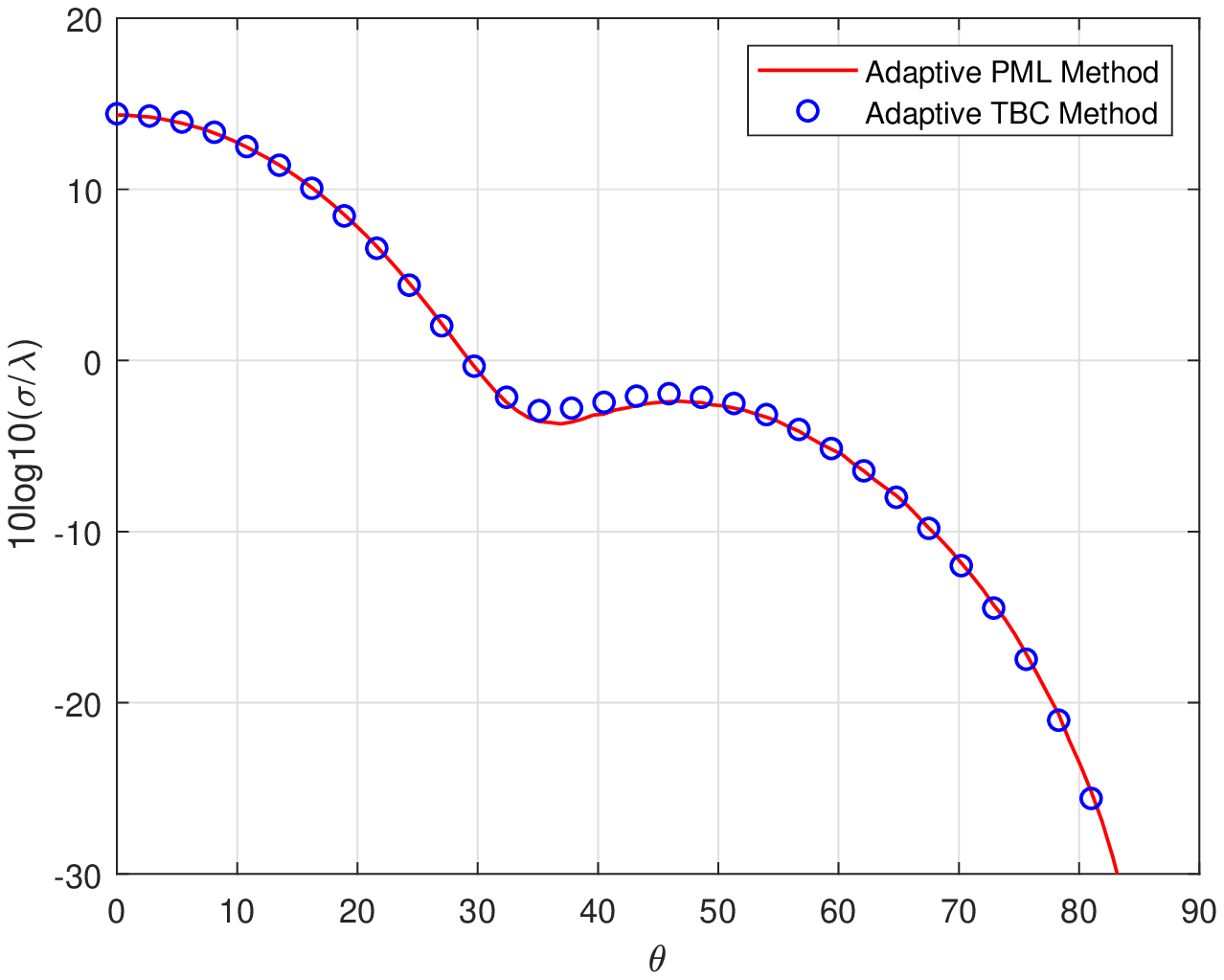}
\caption{Example 2: (left) the cavity geometry; (right) the backscatter RCS
by using the adaptive PML method and the adaptive TBC method.}
\label{TM_Example2}
\end{figure}

\begin{figure}
\centering
\includegraphics[width=0.45\textwidth]{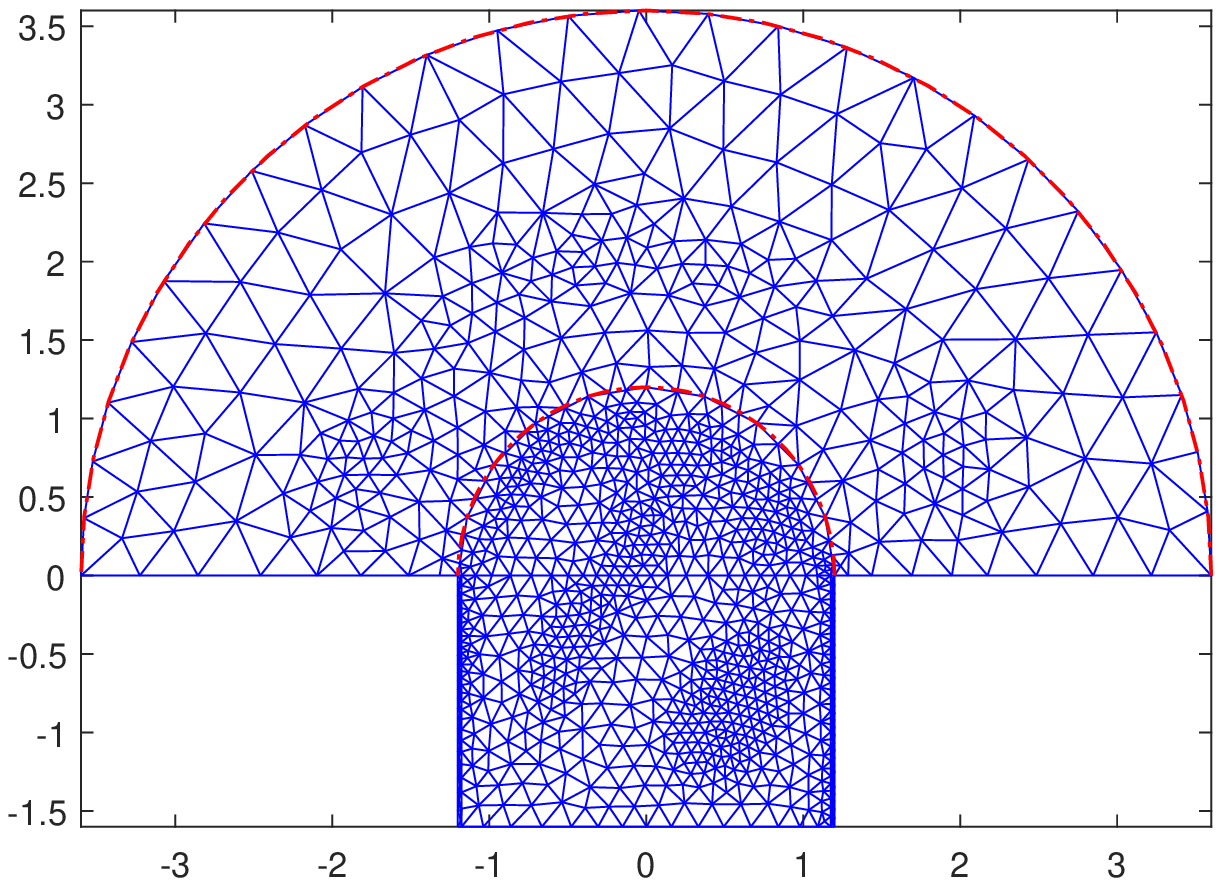}\quad
\includegraphics[width=0.45\textwidth]{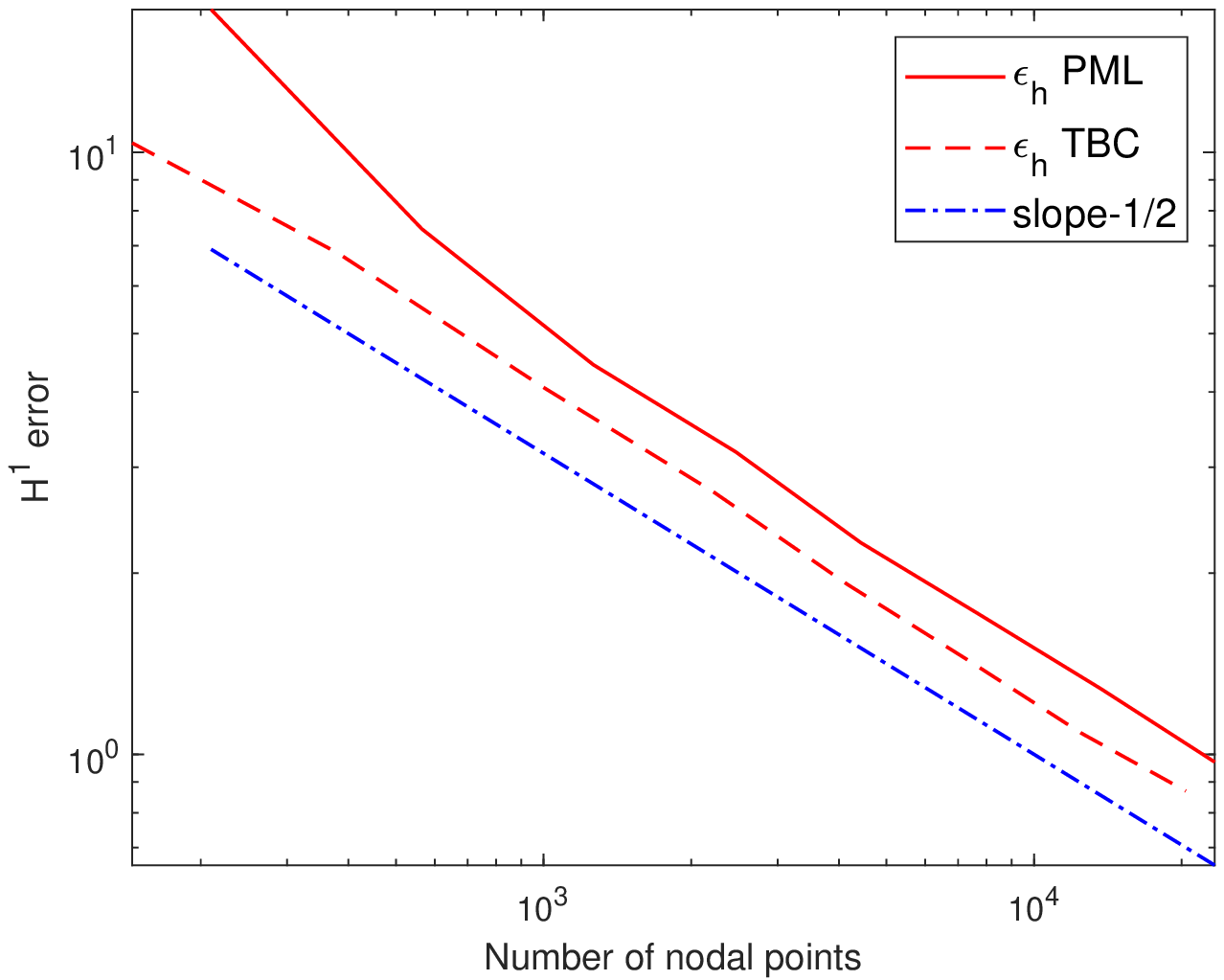}
\caption{Example 2: (left) the adaptive mesh after 2 iterations with a total
number of nodal points $1263$; (right) the quasi-optimality of the a posteriori
error estimates for both of the PML and TBC methods.}
\label{TM_Example2_2}
\end{figure}

\begin{figure}
\centering
\includegraphics[width=0.45\textwidth]{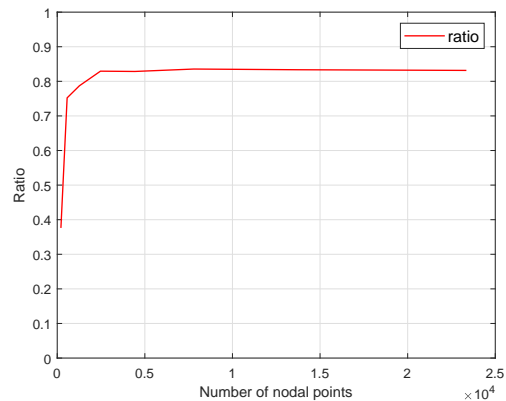}
\caption{Example 2: the ratio of ${\rm DoF}_h$ between the physical domain and
the whole computational domain.}
\label{TM_Example2_3}
\end{figure}

\subsection{Example 3}

This example is still concerned with the TM polarization but some part of the
structure for the cavity sticks out above the ground plane. The width and depth
of the base cavity is $1.2 \lambda$ and $0.8 \lambda$, respectively. There are
two thin rectangular PEC humps in the middle of the cavity. Their width is
$\frac{1}{20}\lambda$ and height is $\frac{16}{15} \lambda$ and $\frac{8}{15}
\lambda$, respectively.  The geometry of the cavity is shown in the left hand
side of Figure \ref{TM_Example3}. The backscatter RCS is computed by using the
adaptive PML and the adaptive TBC method. We also use the red solid line for
the PML method and blue circles for the TBC method. Using the incident angle
$\theta=\pi/4$ as an example, we show the refined mesh after three iterations
with the number of nodal points 1261 and the a posteriori error estimates in
Figure \ref{TM_Example3_2}. The adaptive PML method is able to refined meshes
around the corners of the cavity where the solution has a singularity. The
quasi-optimality is also obtained for the a posteriori error estimates. As a
comparison, we show the a posteriori error estimates for the adaptive TBC method
in the red dashed line. The observation of the a posteriori error estimates for
both methods is the same as that in Examples 1 and 2. Finally, Figure
\ref{TM_Example3_3} shows the ratio of ${\rm DoF}_h$ between the physical
domain and the whole computational domain. Once again, the example confirms
the effectiveness of the adaptive method.

\begin{figure}
\centering
\includegraphics[width=0.45\textwidth]{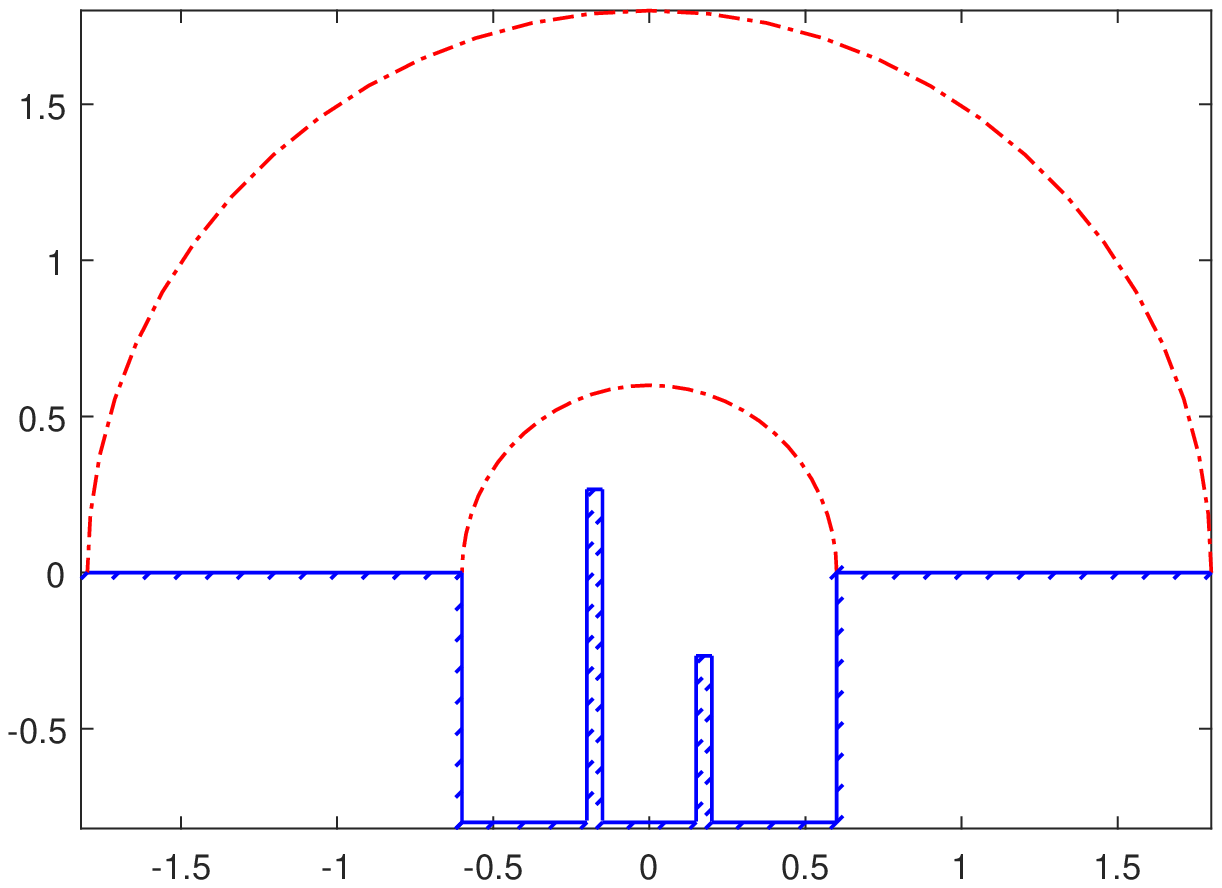}\quad
\includegraphics[width=0.45\textwidth]{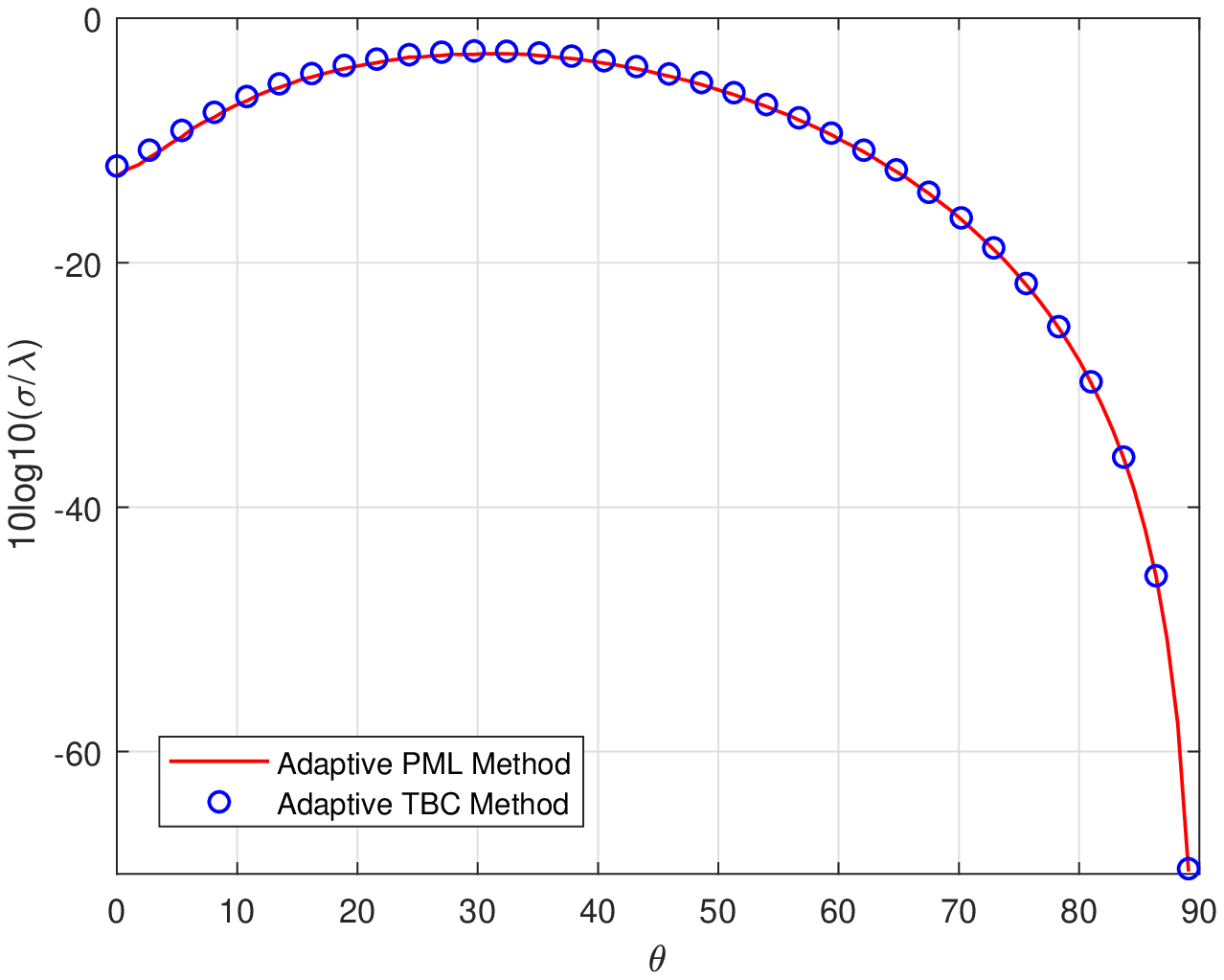}
\caption{Example 3: (left) the cavity geometry; (right) the backscatter RCS by
using the adaptive PML method and the adaptive TBC method.}
\label{TM_Example3}
\end{figure}

\begin{figure}
\centering
\includegraphics[width=0.45\textwidth]{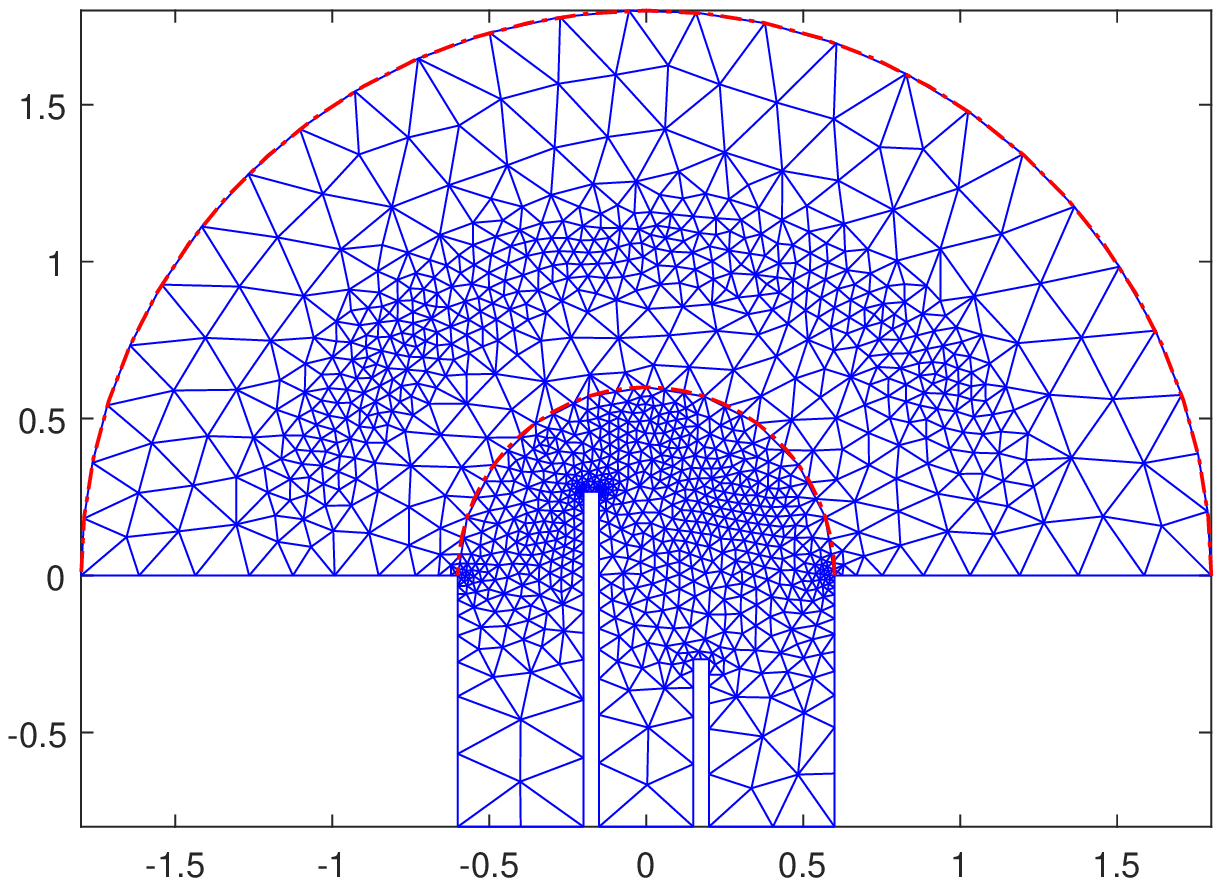}\quad
\includegraphics[width=0.45\textwidth]{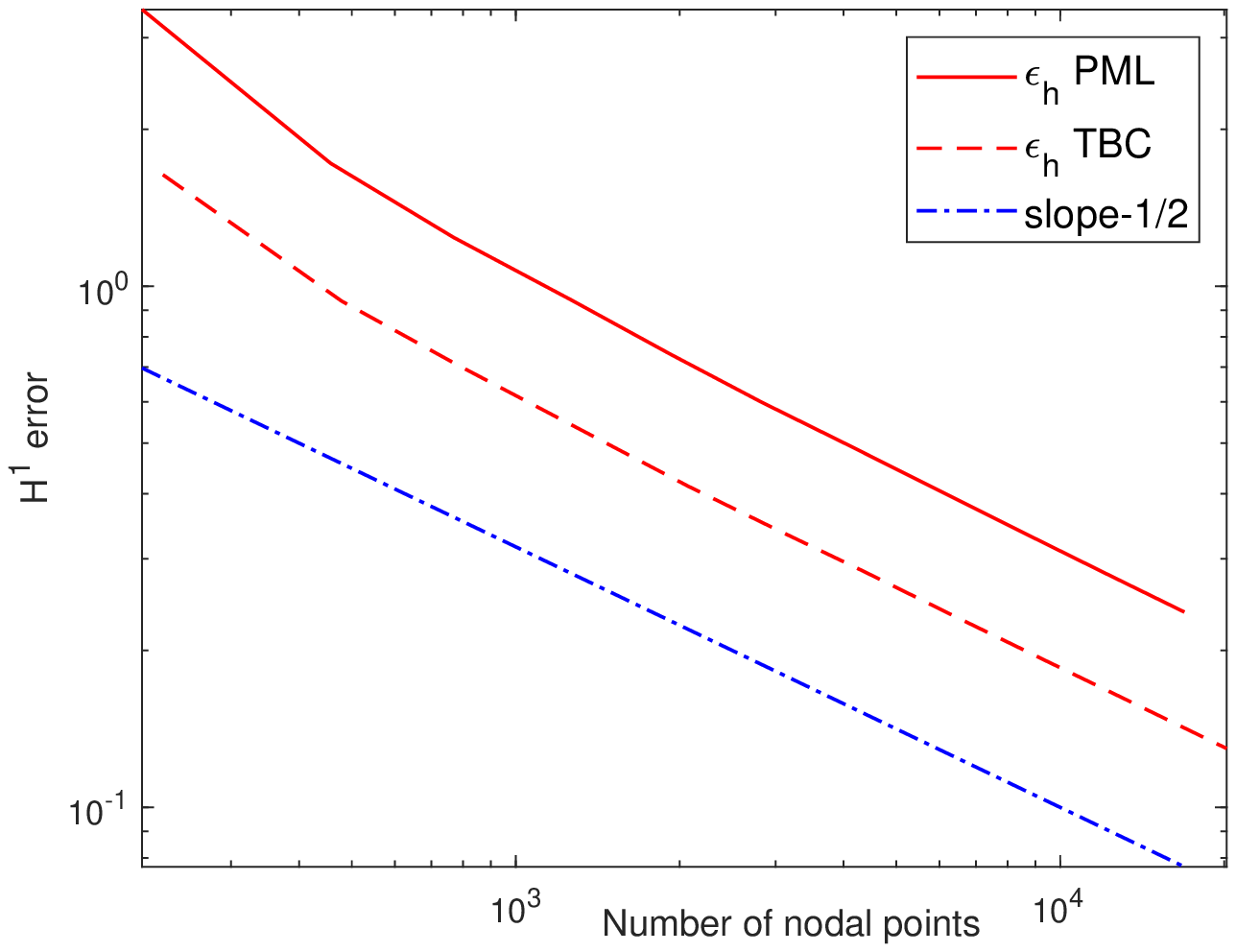}
\caption{Example 3: (left) the adaptive mesh after 3 iterations with a total
number of nodal points $1261$; (right) the quasi-optimality of the a posteriori
error estimates for both of the PML and TBC methods.}
\label{TM_Example3_2}
\end{figure}

\begin{figure}
\centering
\includegraphics[width=0.45\textwidth]{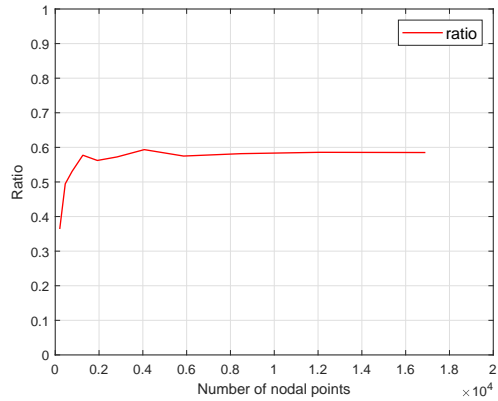}
\caption{Example 3: the ratio of ${\rm DoF}_h$ between the physical domain and
the whole computational domain.}
\label{TM_Example3_3}
\end{figure}

\subsection{Example 4}

In this example, we consider the TE polarized cavity scattering problem. The
cavity is a rectangle with a fixed width 0.025 m and a fixed depth 0.015 m. The
cavity is empty with no filling materials. Instead of considering the
illumination by a plane wave with a fixed frequency, we compute the backscatter
RCS with the frequency ranging from 2 GHz to 18 GHz. Correspondingly, the range
of the aperture of cavity is from $\frac{1}{6}\lambda$ to $1.5 \lambda$. The
incident angle is fixed to be $\frac{4}{9}\pi$. Figure \ref{TE_Example4} shows
the backward RCS by using the adaptive PML and the adaptive TBC method, where
the red solid line and blue circles show their results, respectively. The
stopping criterion is that the mesh refinement is stopped when the number of
nodal points is over 25000. Once again, both methods are consistent with each
other very well.

\begin{figure}
\centering
\includegraphics[width=0.45\textwidth]{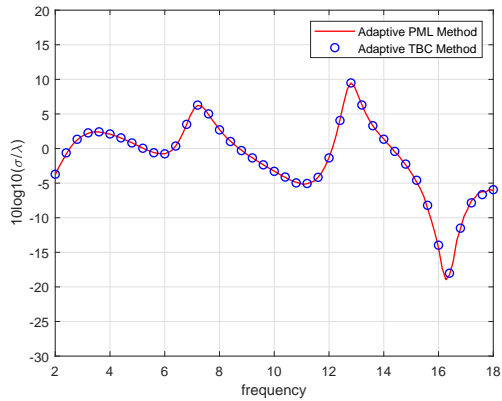}
\caption{Example 4: the backscatter RCS by using the adaptive PML method
and the adaptive TBC method.}
\label{TE_Example4}
\end{figure}

\section{Conclusion}

In this paper, we have presented an adaptive finite element PML method for
solving the open cavity scattering problems. The a posteriori error analysis
is carried out for both of the TM and TE polarizations. In each polarization,
the estimate takes account of the finite element discretization error and
the truncation error of PML method. The latter is shown to decay exponentially
with respect to the PML medium parameter and the thickness of the layer. A
possible future work is to extend our analysis to the adaptive
finite element PML method for solving the three-dimensional cavity
scattering problem, where the wave propagation is governed by Maxwell's
equations.


\begin{thebibliography}{00}

\bibitem{AC91}
M. Ainsworth and A. W. Craig, A posteriori error estimators in the finite
element method, Numer. Math., 60 (1991), 429--463.

\bibitem{AO93}
M. Ainsworth and J. T. Oden, A unified approach to a posteriori error estimation
using element residual methods, Numer. Math., 65 (1993), 23--50.

\bibitem{BA73}
I. Babu\v{s}ka and A. Aziz, Survey lectures on Mathematical Foundation
of the Finite Element Method, in the Mathematical Foundations of the Finite
Element Method with Application to the Partial Differential Equations,
ed. by A. Aziz, Academic Press, New York, 1973, 5--359.

\bibitem{BR78}
I. Babu\v{s}ka and W. C. Rheinboldt, Error estimates for adaptive
finite element, SIAM J. Numer. Anal., 15 (1978), 736--754.

\bibitem{BW85}
R. E. Bank and A. Weiser, Some a posteriori error estimators for elliptic
partial differential equations, Math. Comp., 44 (1985), 283--301.

\bibitem{BW05}
G. Bao and H. Wu, Convergence analysis of the PML problems for time-harmonic
Maxwell's equations, SIAM. J. Numer. Anal., 43 (2005), 2121--2143.

\bibitem{Berenger94}
J.-P. B\'{e}renger, A perfectly matched layer for the absorption of
electromagnetic waves, J. Comput. Phys., 114 (1994), 185--200.

\bibitem{BP07}
J. Bramble and J. Pasciak, Analysis of a finite PML approximation for the three
dimensional time-harmonic Maxwell's and acoustic scattering problems, Math.
Comp., 76 (2007), 597--614.

\bibitem{CC08}
J. Chen and Z. Chen, An adaptive perfectly matched layer technique for 3-D
time-harmonic electromagnetic scattering problems, Math. Comp., 76 (2008),
673--698.

\bibitem{CD01}
Z. Chen and S. Dai, Adaptive Galerkin methods with error control for a dynamical
Landau model in superconductivity, SIAM J. Numer. Anal., 38 (2001), 1961--1985.

\bibitem{CD02}
Z. Chen and S. Dai, On the efficiency of adaptive finite element methods for
with discontinuous coefficients, SIAM J. Sci. Comput., 24 (2002), 443--462.

\bibitem{CL05}
Z. Chen and X. Liu, An adaptive perfectly matched layer technique for
time-harmonic scattering problems, SIAM. J. Numer. Anal., 43 (2005), 645--671.

\bibitem{CW03}
Z. Chen and H. Wu, An adaptive finite element method with perfectly matched
absorbing layer for the wave scattering by periodic structures, SIAM. J. Numer.
Anal., 41 (2003), 799--826.

\bibitem{CZ17}
Z. Chen and W. Zheng, PML Method for electromagnetic scattering problem in a
two-layered medium, SIAM. J. Numer. Anal., 55 (2017), 2050--2084.

\bibitem{CM98Op}
F. Collino and P. Monk, Optimizing the perfectly matched layer,
Comput. Methods Appl. Mech. Engrg., 164 (1998), 157--171.

\bibitem{CM98}
C. Collino and P. Monk, The perfectly matched layer in curvilinear coordinates,
SIAM, J. Sci. Comput., 19(1998), 2061--2090.


\bibitem{HSZ03}
T. Hohage, F. Schmidt and L. Zschiedrich, Solving time-harmonic scattering
problems based on the pole condition, II. Convergence of the PML methods, SIAM
J. Appl. Math., 35 (2003), 547--560.

\bibitem{j-02}
J. Jin, The Finite Element Method in Electromagnetics, Wiley \& Son, New York,
2002.

\bibitem{LS98}
M. Lassas and E. Somersalo, On the existence and convergence of the solution of
the PML equations, Computing, 60 (1998), 229--241.


\bibitem{L18}
P. Li, A survey of open cavity scattering problems, J. Comp. Math., 36
(2018), 1--16.

\bibitem{LWZ12}
P. Li, H. Wu, and W. Zheng, An overfilled cavity problem for Maxwell's
equations, Math. Meth. Appl. Sci., 15 (2012), 1951--1979.


\bibitem{M98}
P. Monk, A posteriori error indicators for Maxwell's equations, J. Comput. Appl.
Math., 100 (1998), 173--190.

\bibitem{MS98}
P. Monk and E. Soli, The adaptive computation of far-field patterns estimation
of linear functionals, SIAM J. Numer. Anal., 36 (1998), 251--274.

\bibitem{TC97}
F. Teixera and W. Chew, Systematic derivation of anisotropic PML absorbing media
in cylindrical and spherical coordinates, IEEE Microw. Guided Wave Lett.,
7 (1997), 371--373.

\bibitem{TY98}
E. Turkel and A. Yefet, Absorbing PML boundary layers for wave-like equations,
Appl. Numer. Math., 27 (1998), 533--557.

\bibitem{S74}
A. H. Schatz, An observation concerning Ritz-Galerkin methods with indefinite
bilinear forms, Math. Comp., 28 (1974), 959--962.

\bibitem{W06}
A. Wood, Analysis of electromagnetic scattering from an overfilled cavity in the
ground plane, J. Comput. Phys., 215 (2006), 630--641.

\bibitem{YBL20}
X. Yuan, G. Bao, and P. Li, An adaptive finite element DtN method for the open
cavity scattering problems, CSIAM Trans. Appl. Math., to appear.

\end{thebibliography}
\end{document}